\documentclass[12pt,centertags,oneside]{amsart}
\usepackage{amsmath,amstext,amsthm,amscd,typearea}
\usepackage{amssymb}
\usepackage{fancyhdr}
\usepackage[mathscr]{eucal}
\usepackage{mathrsfs}
\usepackage{concmath}
\usepackage{charter}

\usepackage[a4paper,width=16.2cm,top=3cm,bottom=3cm]{geometry}

\numberwithin{equation}{section}
\newtheorem{theorem}{Theorem}[section]
\newtheorem{lemma}[theorem]{Lemma}
\newtheorem{proposition}[theorem]{Proposition}
\newtheorem{corollary}[theorem]{Corollary}

\theoremstyle{definition}
\newtheorem{definition}[theorem]{Definition}

\newtheorem{remark}[theorem]{Remark}
\newtheorem{notation}[theorem]{Notation}
\newtheorem{condition}[theorem]{Condition}

\numberwithin{equation}{section}

\newcommand{\field}[1]{\mathbb{#1}}
\newcommand{\Z}{\field{Z}}
\newcommand{\R}{\field{R}}
\newcommand{\C}{\field{C}}
\newcommand{\N}{\field{N}}

\newcommand{\cali}[1]{\mathscr{#1}}
\newcommand{\cC}{\cali{C}} 
\newcommand{\cO}{\cali{O}} 
 
 \newcommand{\cL}{\cali{L}}
 
\newcommand{\cK}{\cali{K}} \newcommand{\cT}{\cali{T}}
\newcommand{\cF}{\cali{F}}
\newcommand{\calig}[1]{\mathcal{#1}}

 \newcommand\mO{\calig{O}}
\newcommand\mQ{\calig{Q}} 
 \newcommand{\cP}{\cali{P}}
\newcommand\mK{\calig{K}}

\newcommand{\boldsym}[1]{\boldsymbol{#1}}
\newcommand\bb{\boldsym{b}}

\def\Im{{\rm Im}}

\newcommand{\imat}{\sqrt{-1}}
\DeclareMathOperator{\End}{End}

\DeclareMathOperator{\Ker}{Ker}

\DeclareMathOperator{\Dom}{Dom}

\DeclareMathOperator{\rank}{rk}
\DeclareMathOperator{\Id}{Id}
\DeclareMathOperator{\supp}{supp}
\DeclareMathOperator{\tr}{Tr}

\DeclareMathOperator{\td}{Td}

\DeclareMathOperator{\ch}{ch}
\DeclareMathOperator{\Ric}{Ric}
\DeclareMathOperator{\spec}{Spec}
\DeclareMathOperator{\pr}{pr} 

\newcommand{\db}{\overline\partial}
\newcommand{\spin}{$\text{spin}^c$ }
\newcommand{\norm}[1]{\lVert#1\rVert}
\newcommand{\abs}[1]{\lvert#1\rvert}
\newcommand{\om}{\omega}
\newcommand{\ov}{\overline}
\newcommand{\var}{\varepsilon}
\newcommand{\wi}{\widetilde}

\newcommand{\comment}[1]{}

\allowdisplaybreaks

\begin{document}

\title{Berezin-Toeplitz Quantization and its kernel expansion}

\date{\today}
\author{Xiaonan Ma}
\address{Universit{\'e} Paris Diderot - Paris 7,
UFR de Math{\'e}matiques, Case 7012,
Site Chevaleret,
75205 Paris Cedex 13, France}
\email{ma@math.jussieu.fr}
\thanks{Partially supported by Institut Universitaire de France}
\author{George Marinescu}
\address{Universit{\"a}t zu K{\"o}ln,  Mathematisches Institut,
    Weyertal 86-90,   50931 K{\"o}ln, Germany\\
    \& Institute of Mathematics `Simion Stoilow', Romanian Academy,
Bucharest, Romania}
\thanks{Partially supported by DFG funded projects SFB/TR 12 and MA 2469/2-1}
\email{gmarines@math.uni-koeln.de}


\begin{abstract}
We survey recent results \cite{MM07,MM08a,MM08b,MM10} 
about the asymptotic expansion of Toeplitz operators and their kernels, 
as well as Berezin-Toeplitz quantization. We deal in particular with calculation 
of the first coefficients of these expansions.
\end{abstract}
\maketitle

\tableofcontents


\section{Introduction}

The aim of the geometric quantization theory of Kostant and Souriau 
is to relate the classical observables 
(smooth functions) on a phase space (a symplectic manifold) to 
the quantum observables (bounded linear operators) on the quantum space 
(sections of a line bundle). 
Berezin-Toeplitz
quantization is a particularly efficient version of the geometric quantization theory
\cite{BFFLS,Berez:74,Fedo:96,Kos:70,Sou:70}.
Toeplitz operators and more generally Toeplitz structures were introduced in 
geometric quantization by
Berezin \cite{Berez:74} and Boutet de Monvel-Guillemin \cite{BouGu81}.
Using the analysis of Toeplitz structures 
\cite{BouGu81}, Bordemann-Meinrenken-Schlichenmaier \cite{BMS94} 
and Schlichenmaier \cite{Schlich:00} gave asymptotic expansion for
the composition of Toeplitz operators in the K\"ahler case.

The expansions we will be considering are asymptotic expansions relative 
to the high power $p$ of the quantum line bundle. The limit $p\to\infty$ 
is interpreted as semi-classical limit process, where the role of 
the Planck constant is played by $\hbar=1/p$. 

The purpose of this paper is to review some methods and results concerning 
Berezin-Toeplitz quantization which appeared in the recent articles 
\cite{MM08a,MM08b,MM10} and in the book \cite{MM07}.
Our approach is based on kernel calculus and the off-diagonal asymptotic 
expansion of the Bergman kernel.
This method allows not only to derive the asymptotic expansions of 
the Toeplitz operators but also to calculate the first coefficients of 
the various expansions. Since the formulas for the coefficients encode 
geometric data of the manifold and prequantum bundle they found extensive 
and deep applications in the study of K\"ahler manifolds 
(see e.g.\ \cite{D09,DouKle10,DouKle08,Fine08,Fine10,LM07,Wangl03,Wang05}, 
to quote just a few).
We will also twist the powers of the prequantum bundle with a fixed 
auxiliary bundle, so the formulas for the coefficients also mirror the curvature 
of the twisting bundle.   

The paper is divided in three parts, treating the quantization of 
K\"ahler manifolds, of K\"ahler orbifolds and finally of symplectic manifolds.

In these notes we do not attempt to be exhaustive, neither in the choice 
of topics, nor in what concerns the references.
For previous work on Berezin-Toeplitz star products in special cases
see \cite{CaGuRa:90,MoOr:83}. The paper \cite{MZ08} contains a detailed 
study of Bergman kernels and Toeplitz operators on K\"ahler and 
symplectic manifolds in the presence of a Hamiltonian action of 
a compact connected Lie group.  
We also refer the reader to the survey articles \cite{AE05,Ma10,Schlich:10} 
for more information for the Berezin-Toeplitz quantization and 
geometric quantization. The survey \cite{Ma10} gives a review 
in the context of K\"ahler and symplectic manifolds and explores 
the connections to symplectic reduction.


\section{Quantization of K\"ahler manifolds}
In this long section we explain our approach to Berezin-Toeplitz quantization of symplectic manifolds by specializing to the K\"ahler case. The method we use is then easier to follow and the coefficients of the asymptotic expansions have accurate expressions in terms of curvatures of the underlying manifold.

In Section 2.1, we review the definition of the Bergman projector, introduce the Toeplitz operators and their kernels. 

In  Section 2.2, we describe the spectral gap of the Kodaira-Laplace operator. On one hand, this implies the Kodaira-Serre vanishing theorem 
and the fact that for high powers of the quantum line bundle the cohomology concentrates in degree zero. 
On the other hand, the spectral gap provides the natural framework for the asymptotic expansions of the Bergman and Toeplitz kernels. 

In  Section 2.3 we describe the model operator, its spectrum and the kernel of its spectral projection on the lowest energy level.
The expansion of the Bergman kernel, which we study in Section  2.4, is obtained by a localization and rescaling technique due to Bismut-Lebeau \cite{BL91}, and reduces the problem to the model case.

With this expansion at hand, we formulate the expansion of the Toeplitz kernel in Section 2.5. Moreover, we observe in Section 2.6 that these expansion characterizes the Toeplitz operators and this characterization implies the expansion of the product of two Toeplitz operators and the existence of the Berezin-Toeplitz star product.

In Section 2.7, we explain how to apply the previous results when the Riemannian metric used to define the Hilbertian structure on the space of sections is arbitrary.

In Section 2.8, we turn to the general situation of complete K\"ahler manifolds and show how to apply the introduced method in this case.


\subsection{Bergman projections, Toeplitz operators, and their kernels}\label{s2.1}

We consider a complex manifold $(X,J)$ with complex structure $J$, and complex dimension $n$.
Let $L$ and $E$ be two holomorphic vector bundles on $X$. We assume that $L$ is a line bundle i.e.\ $\operatorname{rk}\,L=1$. The bundle $E$ is an auxiliary twisting bundle.
It is interesting to work with a twisting vector bundle $E$ for several reasons. For example, one has to deal with $(n,0)$-forms with values in $L^{p}$, so one sets $E=\Lambda^{n} (T^{*(1,0)}X)$. From a physical point of view, the presence of $E$ means a quantization of a system with several degrees of internal freedom.

We fix Hermitian metrics $h^L$, $h^E$ on $L$, $E$. Let $g^{TX}$ be a $J$-invariant Riemannian metric on $X$, i.e., $g^{TX}(Ju,Jv)= g^{TX}(u,v)$ for all $x\in X$ and $u,v\in T_xX$. The Riemannian volume form of $g^{TX}$ is denoted by $dv_X$.
On the space of smooth sections with compact support $\cC^{\infty}_0(X,L^p\otimes E)$ we introduce the $L^2$-scalar product associated to the metrics $h^L$, $h^E$ and
the Riemannian volume form $dv_X$ by
\begin{equation}\label{lm2.0}
\big\langle s_1,s_2 \big\rangle =\int_X\big\langle s_1(x),
s_2(x)\big\rangle_{L^p\otimes E}\,dv_X(x)\,.
\end{equation}
The completion of $\cC^{\infty}_0(X,L^p\otimes E)$ with respect to \eqref{lm2.0} is denoted as usual by $L^{2}(X,L^p\otimes E)$. 
We consider the space of holomorphic $L^{2}$ sections:
\begin{equation}\label{lm2.02}
H^{0}_{(2)}(X,L^p\otimes E):=\big\{s\in L^{2}(X,L^p\otimes E) : \text{$s$ is holomorphic}\big\}\,.
\end{equation}

Let us note an important property of the space $H^{0}_{(2)}(X,L^p\otimes E)$, which follows from the Cauchy estimates for holomorphic functions. Namely, for every compact set
$K\Subset X$ there exists $C_K>0$ such that
\begin{equation}\label{b1.1}
\sup_{x\in K}|S(x)|\leqslant C_K\|S\|_{L^2}\,,\quad\text{for all $S\in H^{0}_{(2)}(X,L^p\otimes E)$\,.}
\end{equation}
We deduce that $H^{0}_{(2)}(X,L^p\otimes E)$ is a closed subspace of $L^2(X,L^p\otimes E)$; one can also show that $H^{0}_{(2)}(X,L^p\otimes E)$ is separable (cf.\ \cite[p.\,60]{Weil:58}). 
\begin{definition} \label{almt2.1b}
The \emph{Bergman projection} is the orthogonal projection
\[
P_{p}:L^{2}(X,L^p\otimes E)\to H^{0}_{(2)}(X,L^p\otimes E)\,.
\]
\end{definition}

In view of \eqref{b1.1}, the Riesz representation theorem shows that for a fixed $x\in X$ there exists $P_p(x,\cdot)\in L^2(X,(L^{p}\otimes E)_x\otimes(L^{p}\otimes E)^{*})$ such that
\begin{equation}\label{lm2.01a}
S(x)=\int_{X}P_{p}(x,x') S(x') dv_{X}(x')\,,\quad\text{for all $S\in H^{0}_{(2)}(X,L^p\otimes E)$\,.}
\end{equation}

\begin{definition} \label{almt2.1c}
The section $P_{p}(\cdot,\cdot)$ of $(L^{p}\otimes E)\boxtimes(L^{p}\otimes E)^{*}$ over $X\times X$ is called the \emph{Bergman kernel} of $L^{p}\otimes E$.
\end{definition}
Set $d_p := \dim H^{0}_{(2)}(X,L^p\otimes E)\in\N\cup\{\infty\}$.
Let $\{S^p_i\}_{i=1}^{d_p}$ be any orthonormal basis of
$H^{0}_{(2)}(X,L^p\otimes E)$ with respect to the inner
product \eqref{lm2.0}. Using the estimate \eqref{b1.1} we can show that
\begin{equation} \label{bk2.4}
P_p(x,x')= \sum_{i=1}^{d_p} S^p_i (x) \otimes (S^p_i(x'))^*
\in (L^p\otimes E)_x\otimes (L^p\otimes E)_{x'}^*\,,
\end{equation}
where the right-hand side converges on every compact together with all its derivatives (see e.g.\ \cite[p.\,62]{Weil:58}).
Thus $P_{p}(\cdot,\cdot)\in \cC^{\infty}(X\times X,(L^{p}\otimes E)\boxtimes(L^{p}\otimes E)^{*})$.
It follows that 
\begin{equation}\label{lm2.01}
(P_{p}S)(x)=\int_{X}P_{p}(x,x') S(x') dv_{X}(x')\,,\quad\text{for all $S\in L^2(X,L^p\otimes E)$\,.}
\end{equation}
that is, $P_{p}(\cdot,\cdot)$ is the integral kernel of the Bergman projection $P_p$. 

\noindent
We recall that a Carleman kernel (see e.g.\ \cite{HalSu}) is a (measurable) section 
\[
T(\cdot,\cdot):X\times X\to (L^p\otimes E)\boxtimes (L^p\otimes E)^*
\] 
such that 
\begin{gather*}
T(x,\cdot)\in L^2(X,(L^p\otimes E)_x\otimes(L^p\otimes E)^*)\,,\:\text{for almost all $x\in X$},\\
T(\cdot,x')\in L^2(X,(L^p\otimes E)\otimes(L^p\otimes E)^*_{x'})\,,\:\text{for almost all $x'\in X$}.
\end{gather*}
A bounded linear operator $T$ on $L^2(X,L^p\otimes E)$ is called
Carleman operator if there exists a Carleman kernel $T(\cdot,\cdot)$ such that 
\begin{equation}\label{lm2.01c}
(TS)(x)=\int_{X}T(x,x') S(x') dv_{X}(x')\,,\quad\text{for all $S\in L^2(X,L^p\otimes E)$.}
\end{equation}
Note that $P_p(x,x')=P_p(x',x)^*$, thus $P_p(x,x')$ is a Carleman kernel and $P_p$ is a Carleman operator.
Let $T_1$, $T_2$ be two Carleman operators on $L^2(X,L^p\otimes E)$. Then the composition $T_1\circ T_2$ is a Carleman operator with kernel
\[
(T_1\circ T_2)(x,x'')=\int_X T_1(x,x')T_1(x',x'')\,dv_X(x')\,. 
\]

The Bergman kernel represents the local density of the space of holomorphic sections and is a very efficient tool to study properties of holomorphic sections.
It is an ``objet souple'' in the sense of Pierre Lelong, that is, it interpolates between the rigid objects of complex analysis and the flexible ones of real analysis.

Note that $P_p(x,x)\in\End(E)_x$, since $\End(L^p)=\C$. Using \eqref{bk2.4} and the formula 
$\tr_E \big[S^p_i (x) \otimes (S^p_i(x))^*\big]=\vert S^p_i (x)\vert^2$, we obtain immediately
\begin{equation} \label{bk2.41}
d_p=\int_X \tr_E P_p(x,x)\,dv_X(x)\,.
\end{equation}

\begin{definition}\label{ber-toe-def}
For a bounded section $f\in\cC^{\infty}(X,\End(E))$, set 
\begin{equation}\label{toe2.4}
T_{f,\,p}:L^2(X,L^p\otimes E)\longrightarrow L^2(X,L^p\otimes E)\,,
\quad T_{f,\,p}=P_p\,f\,P_p\,,
\end{equation}
where the action of $f$ is the pointwise multiplication by $f$.
The map which associates to  $f\in \cC^{\infty}(X,\End(E))$
the family of bounded operators $\{T_{f,\,p}\}_p$ on $L^2(X,L^p\otimes E)$ is called
the  {\em Berezin-Toeplitz quantization}\/.
\end{definition}
Note that $T_{f,\,p}$ is a Carleman operator with smooth integral kernel given by
\begin{equation}\label{toe2.5}
T_{f,\,p}(x,x')=\int_X P_p(x,x'')f(x'')P_p(x'',x')\,dv_X(x'')\,.
\end{equation}

For two arbitrary bounded sections
$f,g\in \cC^{\infty}(X,\End(E))$ it is easy to see that $T_{f,\,p}\circ T_{g,\,p}$ is not in general of the form
$T_{fg,\,p}$\,. But we have $T_{f,\,p}\circ T_{g,\,p}\sim T_{fg,\,p}$ asymptotically for $p\to\infty$. In order to explain this we introduce the following more general notion of Toeplitz operator.
\begin{definition}\label{toe-def}
A {\em Toeplitz operator}\index{Toeplitz operator}
is a sequence $\{T_p\}_{p\in\N}$ of linear operators
\begin{equation}\label{toe2.1}
T_{p}:L^2(X,L^p\otimes E)\longrightarrow L^2(X,L^p\otimes E)
\end{equation}
verifying $T_{p}=P_p\,T_p\,P_p$\,,
such that there exist a sequence $g_\ell\in\cC^\infty(X,\End(E))$ such that
for any $k\geqslant0$, there exists $C_k>0$
 with
\begin{equation}\label{toe2.3}
\Big\|T_p-\sum_{\ell=0}^kT_{g_\ell,\,p}\, p^{-\ell}\Big\|
\leqslant C_k\,\, p^{-k-1}\quad \text{ for any } p\in \N^*,
\end{equation}
where $\norm{\,\cdot\,}$ denotes the operator norm on the space of
bounded operators.
The section $g_0$ is called the {\em principal symbol} of $\{T_p\}$.
\end{definition}
\noindent
We express \eqref{toe2.3} symbolically by
\begin{equation}\label{atoe2.1}
T_p= \sum_{\ell=0}^k T_{g_\ell,p}\, p^{-\ell} + \mO(p^{-k-1}).
\end{equation}
If \eqref{toe2.3} holds for any $k\in \N$, then we write
\eqref{atoe2.1} with $k=+\infty$.
One of our goals is to show that $T_{f,\,p}\circ T_{g,\,p}$ is a Toeplitz operator in the sense of Definition
\ref{toe2.3}. This will be achieved by using the asymptotic expansions of the Bergman kernel and of the kernels of the Toeplitz operators.


\subsection{Spectral gap and vanishing theorem}\label{spec-gap}

In order to have a meaningful theory it is necessary that the spaces $H^0_{(2)}(X,L^p\otimes E)$ are as large as possible. In this section we describe conditions when the growth of $d_p=\dim H^0_{(2)}(X,L^p\otimes E)$ for $p\to\infty$ is maximal.

For this purpose we need Hodge theory, so we introduce the Laplace operator.
Let $T^{(1,0)}X$ be the holomorphic tangent bundle on $X$, $T^{(0,1)}X$ the conjugate of $T^{(1,0)}X$ 
and $T^{*(0,1)}X$ the dual bundle of $T^{(0,1)}X$. We denote by $\Lambda^q(T^{*(0,1)}X)$ the bundle of $(0,q)$-forms on $X$ and by
$\Omega^{0,q}(X,F)$ the space of sections of the bundle $\Lambda^q(T^{*(0,1)}X)\otimes F$ over $X$, for some vector bundle $F\to X$.

The Dolbeault operator acting on sections of the holomorphic vector bundle $L^p\otimes E$ gives rise to the Dolbeault complex
\[\Big(\Omega^{0,\bullet}(X,L^p\otimes E), \overline{\partial}^{L^p\otimes E}\Big)\,.\]
 Its cohomology,
called Dolbeault cohomology,
is denoted by $H^{0,\bullet}(X,L^p\otimes E)$.
We denote by $\overline{\partial} ^{L^p\otimes E,*}$ the formal adjoint of $\overline{\partial} ^{L^p\otimes E}$
with respect to the $L^2$-scalar product \eqref{lm2.0}.
Set
\begin{align}\label{lm2.1}
\begin{split}
D_p &= \sqrt{2}\big(\, \overline{\partial}^{L^p\otimes E}
+ \,\overline{\partial}^{L^p\otimes E,*}\,\big)
\,,\\
\square^{L^p\otimes E} &= \tfrac{1}{2}D^2_p =\overline{\partial}^{L^p\otimes E}\,
\overline{\partial}^{L^p\otimes E,*}
+\,\overline{\partial}^{L^p\otimes E,*}\,\overline{\partial}^{L^p\otimes E}.
\end{split}
\end{align}
The operator $\square^{L^p\otimes E}$ is called the \emph{Kodaira-Laplacian}. It acts on
$\Omega ^{0,\bullet}(X, L^p\otimes E)$
and preserves its $\Z$-grading.

Let us consider first that $X$ is a \emph{compact K\"ahler} manifold endowed with a K\"ahler form $\omega$ and $L$ is a \emph{prequantum} line bundle. The latter means that there exists a Hermitian metric $h^L$ such that the curvature $R^L=(\nabla^L)^2$ of the holomorphic Hermitian connection $\nabla^L$ on $(L, h^L)$ satisfies
\begin{equation}\label{prequantum}
\omega=\frac{\sqrt{-1}}{2\pi}R^{L}\,.
\end{equation}
In particular, $L$ is a positive line bundle.

By Hodge theory, the elements of $\Ker(\square^{L^p\otimes E})$,
called \emph{harmonic forms}, represent the Dolbeault cohomology. Namely,
\begin{equation} \label{lm2.3}
\Ker (D_p|_{\Omega ^{0,q}})  =\Ker (D_p^2|_{\Omega ^{0,q}})
\simeq H^{0,q} (X, L^p\otimes E).
\end{equation}
and the spaces $H^{0,q} (X, L^p\otimes E)$ are finite dimensional.
Note that $H^{0,0}(X,L^p\otimes E)$ is the space of holomorphic sections of $L^p\otimes E$, denoted shortly by $H^{0}(X,L^p\otimes E)$. Since $X$ is compact we have $H^0(X,L^p\otimes E)=H^0_{(2)}(X,L^p\otimes E)$ for any Hermitian metrics on $L^p$, $E$ and volume form on $X$.
A crucial tool in our analysis of the Bergman kernel is the following spectral gap of the Kodaira-Laplacian.
\begin{theorem}[{\cite[Th.\,1.1]{BVa89}, \cite[Th.\,1.5.5]{MM07}}]\label{bkt1.1}
There exist constants positive $C_0$, $C_L$ such that for any $p\in \N$ and
any $s\in\Omega^{0, >0}(X,L^p\otimes E)
=\bigoplus_{q\geqslant 1}\Omega^{0,q}(X,L^p\otimes E)$,
\begin{equation}\label{bk1.4}
\norm{D_{p}s}^2_{L^2}\geqslant(2C_0 p -C_L)\norm{s}^2_{L^2}\,.
\end{equation}
Hence
\begin{equation}\label{bk1.41}
\spec(\square_p)\subset\big\{0\big\}\cup\, \big]p\,C_0-\frac{1}{2}C_L,+\infty\big[\,,
\end{equation}
where $\spec(\square_p)$ denotes the spectrum
of the Kodaira Laplacian $\square_p$\,.
\end{theorem}
Theorem \ref{bkt1.1} was first proved by Bismut-Vasserot \cite[Th.\,1.1]{BVa89} using the non-k\"ahlerian Bochner-Kodaira-Nakano formula with torsion due to Demailly, see e.g.\ \cite[Th.\,1.4.12]{MM07} (note that $g^{TX}$ is arbitrary, we don't suppose that it is the metric associated to $\omega$, i.e., $g^{TX}(u,v)=\omega(u,Jv)$ for $u,v\in T_xX$).
By Theorem \ref{bkt1.1}, we conclude:
\begin{theorem}[Kodaira--Serre vanishing Theorem]\label{bkt1.2}
\index{vanishing Theorem!Kodaira--Serre}\label{bkt1.3}
 If $L$ is a positive line bundle, then there exists $p_0>0$
such that for any $p\geqslant p_0$,
\begin{equation}\label{bk1.11}
H^{0,q}(X,L^p\otimes E)= 0 \quad \mbox{\rm for any } \, q>0.
\end{equation}
\end{theorem}
Recall that for a compact manifold $X$ and a holomorphic vector bundle $F$, the Euler number $\chi(X,F)$ is defined by
\begin{equation}\label{alm2.30}
\chi(X,F)= \sum_{q=0}^n (-1)^q \dim H^{0,q}(X,F).
\end{equation}
By the Riemann-Roch-Hirzebruch Theorem \cite[Th.\,14.6]{MM07} we have
\begin{equation}\label{alm2.31}
\chi(X,F)= \int_X \td\big(T^{(1,0)}X\big)\ch(F)\,,
\end{equation}
where $\td$ and $\ch$ indicate the Todd class and the Chern character, respectively.
By the Kodaira-Serre vanishing \eqref{bk1.11},
\begin{equation}\label{alm2.311}
d_p=\dim H^0(X,L^p\otimes E)=\chi(X,L^p\otimes E)\,,\quad p\geqslant p_0\,.
\end{equation}
Therefore, for $p\geqslant p_0$,
\begin{equation}
\begin{split}
&\dim H^{0}(X,L^p\otimes E)=\int_X \td\big(T^{(1,0)}X\big)\ch(L^p\otimes E) \\
&=\rank (E) \int_X \frac{c_1(L)^n}{n!}\, p^n
+ \int_X \Big(c_1(E) + \frac{\rank (E)}{2} c_1(T^{(1,0)}X)\Big)
\frac{c_1(L)^{n-1}}{(n-1)!}\, p^{n-1}\\
 &\qquad+ \cO(p^{n-2}). 
  \end{split}
\end{equation}
Note that the first Chern class $c_1(L)$ is represented by $\omega$ (see \eqref{prequantum}) and $c_1(E)$
is represented by $\frac{\imat}{2\pi}\tr [R^E]$. As a conclusion we have:
\begin{theorem}
Let $(X,\omega)$ be a compact K\"ahler manifold and let $(L,h^L)$ be a prequantum line bundle satisfying \eqref{prequantum}. Then $d_p$ is a polynomial of degree $n$ with positive leading term $\frac{1}{n!}\int_X c_1(L)^n$ {\rm(}the volume of the manifold $(X,\omega)${\rm)}.
\end{theorem}
Let us consider now the general situation of a (non-compact) complex manifold $(X,J)$. As before we are given a Hermitian metric on $X$, that is, a $J$-compatible Riemannian metric $g^{TX}$. We denote by $\Theta$ the associated $(1,1)$-form, i.e.,
$\Theta(u,v)=g^{TX}(Ju,v)$, for all $x\in X$ and $u,v\in T_xX$. We say that the Hermitian manifold $(X,\Theta)$ is complete if the Riemannian metric $g^{TX}$ is complete. Consider further a Hermitian holomorphic vector bundle $(F,h^F)\to X$. Let us denote by $\Omega^{0,q}_{(2)}(X,F):=L^2(X,\Lambda^{q}(T^{*(0,1)}X)\otimes F)$.
We have the complex of closed, densely defined operators
\begin{equation}\label{ell0-1}
\Omega^{0,q-1}_{(2)}(X,F)\stackrel{T=\db^F}{\longrightarrow}\Omega^{0,q}_{(2)}(X,F)
\stackrel{S=\db^F}{\longrightarrow}\Omega^{0,q+1}_{(2)}(X,F)\,,
\end{equation}
where $T$ and $S$ are the maximal extensions of $\db^F$, i.e.,
\[\Dom(\db^F)=\{s\in{\Omega^{0,\bullet}_{(2)}(X,F):
\db^F s\in\Omega^{0,\bullet}_{(2)}(X,F)}\}\]
where $\db^F s$ is calculated in the sense of distributions.
Note that $\Im(T)\subset\Ker(S)$, so $ST=0$.
The $q$-th
\emph{$L^2$ Dolbeault cohomology}
is defined by
\begin{equation}\label{ell-coh}
{H}^{0,\,q}_{(2)}(X,F)
:=\frac{\Ker(\db^F)\cap \Omega^{0,q}_{(2)}(X,F)}{\Im(\db^F)\cap \Omega^{0,q}_{(2)}(X,F)}\;.
\end{equation}
Consider the quadratic form $Q$ given by
\begin{equation}\label{ell2,1}
\begin{split}
&\Dom(Q):=\Dom(S)\cap\Dom(T^*), \\
Q(s_1,s_2)=&\langle S s_1,S s_2\rangle+\langle T^*s_1,T^*s_2\rangle\,,
\quad\text{for  } s_1,s_2\in \Dom(Q).
\end{split}
\end{equation}
where $T^*$ is the Hilbertian adjoint of $T$. For the following result due essentially to Gaffney one may consult
\cite[Prop.\,3.1.2, Cor.\,3.3.4]{MM07}.
\begin{lemma}
Assume that the Hermitian manifold $(X,\Theta)$ is complete. Then the Kodaira-Laplacian $\square^F:\Omega^{0,\bullet}_{0}(X,F)\to\Omega^{0,\bullet}_{(2)}(X,F)$
is essentially self-adjoint. Its associated quadratic form is the form $Q$ given by \eqref{ell2,1}.
\end{lemma}
We denote by $R^{\det}$ the curvature of the holomorphic Hermitian
connection $\nabla^{\det}$ on $K_X^*=\det (T^{(1,0)}X)$.
We have the following generalization of Theorems \ref{bkt1.1} and \ref{bkt1.2}.
\begin{theorem}[{\cite[Th.\,6.1.1]{MM07}, \cite[Th.\,3.11]{MM08a}}] \label{noncompact0}
Assume that $(X,\Theta)$ is a complete Hermitian manifold. Let $(L,h^L)$ and $(E,h^E)$ Hermitian holomorphic vector bundles of rank one and $r$, respectively.
Suppose that there exist $\varepsilon>0$, $C>0$ such that\,{\rm:}
\begin{equation}\label{i}
\sqrt{-1}R^L >\varepsilon\Theta\,,
\quad\,\sqrt{-1}(R^{\det}+R^E)> -C\Theta \Id_E\,,\quad\,
|\partial \Theta|_{g^{TX}}< C,
\end{equation}
Then there exists $C_1>0$ and $p_0\in\N$ such that for $p\geqslant p_0$ the quadratic form $Q_p$ associated to the Kodaira-Laplacian $\square_p:=\square^{L^p\otimes E}$ satisfies
\begin{equation}\label{ell4,1}
Q_p(s,s)\geqslant  C_1 p \,\norm{s}^2_{L^2}\,,
\quad \text{for $s\in\Dom(Q_p)\cap\Omega^{0,q}_{(2)}(X,L^p\otimes{E})$, $q>0$}\,.
\end{equation}
Especially
\begin{equation}\label{ell4,2}
{H}^{0,\,q}_{(2)}(X,L^p\otimes{E})=0\,,\quad\text{for $p\geqslant p_0$, $q>0$}
\end{equation}
and the spectrum $\spec(\square_p)$
of the Kodaira Laplacian $\square_p$ acting on $L^2(X,L^p\otimes E)$ is contained in the
set $\{0\}\cup[\,p\,C_1 ,\infty[$\,.
\end{theorem}
Thus we are formally in a similar situation as in the compact case, that is, the higher $L^2$ cohomology groups vanish. But we cannot invoke as in the compact case the index theorem to estimate the dimension of $L^2$ holomorphic sections of $L^p\otimes E$. Instead we can use an analogue of the local index theorem, namely the asymptotics of the Bergman kernel. Let us denote by $\alpha_1,\ldots,\alpha_n$ the eigenvalues of $\frac{\imat}{2\pi} R^L$ with respect to $\Theta$.
\begin{theorem}[{\cite[Cor.\,3.12]{MM08a}}]\label{noncompact}
Under the hypotheses of Theorem \ref{noncompact0} we have
\begin{equation}\label{ell4,3}
P_p(x,x)=p^n\bb_0(x)+\mO(p^{n-1})\,,\quad p\to\infty\,,
\end{equation}
uniformly on compact sets, where $\bb_0=\alpha_1\ldots\alpha_n\Id_E$. Hence
\begin{equation} \label{ell6}
\liminf_{p\longrightarrow\infty}p^{-n}\dim H^0_{(2)}(X,L^p\otimes{E})\geqslant
\frac{\rank(E)}{n!}\int_X\Big(\tfrac{\sqrt{-1}}{2\pi}R^L\Big)^n .
\end{equation}
\end{theorem}
The asymptotics \eqref{ell4,3} are a particular case of the full asymptotic expansion of the Bergman kernel, see Corollary \ref{bkt2.18}. It can be also deduced with the help of $L^2$ estimates of H\"ormander as done by Tian \cite{T90}.
The estimate \eqref{ell6} shows that $\dim H^0_{(2)}(X,L^p\otimes{E})$ has at least polynomial growth of degree $n$. It follows from Fatou's lemma, applied on $X$ with
the measure $\Theta^n/n!$
to the sequence $p^{-n}\tr_E P_p(x,x)$ which converges pointwise to $\tr_E\bb_0$ on $X$.

\subsection{Model situation: Bergman kernel on $\C^n$}\label{toes1}
We introduce here the model operator, a Kodaira-Laplace operator on $\C^n$, and describe explicitly its spectrum.
The expansion of the Bergman and Toeplitz kernels will be expressed in terms of the kernel of the projection on $\Ker(\cL)$. 
Our whole analysis and calculations are based on the Fourier expansion with respect to the eigenfunctions of $\cL$.
    
Let us consider the canonical real coordinates $(Z_1,\dotsc,Z_{2n})$ on
$\R^{2n}$ and the complex coordinates $(z_1,\dotsc,z_n)$ on $\C^n$.
The two sets of coordinates are linked by the relation
$z_j=Z_{2j-1}+\imat Z_{2j}$, $j=1,\dotsc,n$.
We endow $\C^n$ with the Euclidean metric $g^{T\C ^n}$.
The associated K{\"a}hler form on $\C^n$ is
\[
\om=\frac{\sqrt{-1}}{2}\sum_{j=1}^ndz_{j} \wedge d\ov{z}_{j}\,.
\]
We are interested in the space $(L^2(\R^{2n}),\norm{\,\cdot\,}_{L^2})$
of square integrable functions on $\R^{2n}$ with respect to the Lebesgue measure.
We denote by $dZ=dZ_1\cdots dZ_{2n}$ the Euclidean volume form.
For $\alpha=(\alpha_1, \ldots, \alpha_n)\in \N^{n}$,
$z\in \C^{n}$, put $z^\alpha = z_1^{\alpha_1}\cdots z_{n}^{\alpha_{n}}$.


Let $L=\C$ be the trivial holomorphic line bundle
 on $\C^n$ with the canonical section $\mathbf{1}:\C^n\to L$,  $z\mapsto(z,1)$.
Let $h^L$ be the metric on $L$ defined by
\begin{equation}\label{abk2.65}
|\mathbf{1}|_{h^L}(z) := \exp(\textstyle{-\frac{\pi}{2}\sum_{j=1}^n|z_{j}|^2})=\rho (Z)\quad \text{ for } z\in \C^n\,.
\end{equation}
The space of $L^2$-integrable holomorphic sections
of $L$ with respect to $h^L$ and $dZ$
is the classical Segal-Bargmann space of $L^2$-integrable holomorphic functions
with respect to the volume form $\rho \, dZ$.
It is well-known that $\{z^\beta: \beta\in \N^n\}$
forms an orthogonal basis of this space.

To introduce the model operator $\cL$ we set:
\begin{equation}\label{bk2.67}
\begin{split}
&b_i=-2{\frac{\partial}{\partial z
_i}}+ \pi \overline{z}_i\,,\quad
b^{+}_i=2{\frac{\partial}{\partial\overline{z}_i}}+ \pi
z_i\,,\quad     
\cL=\sum_i b_i\, b^{+}_i\,.
\end{split}\end{equation}
We can interpret
the operator $\cL$ in terms of complex geometry.
Let $\overline{\partial}^{L}$ be the Dolbeault operator acting on $\Omega^{0,\bullet}(\C^n,L)$ and let
$\overline{\partial}^{L,*}$ be its adjoint with respect to the $L^2$-scalar product
induced by $g^{T\C^n}$ and $h^L$.
We have the isometry
$\Omega^{0,\bullet}(\C^n,\C)\to\Omega^{0,\bullet}(\C^n,L)$ given by
$\alpha\mapsto  \rho ^{-1}\alpha$.
If
\[
\square^L=\overline{\partial}^{L,*}\overline{\partial}^{L}
+\overline{\partial}^{L}\overline{\partial}^{L,*}
\]
denotes the Kodaira
Laplacian acting on $\Omega^{0,\bullet}(\C^n,L)$, then
\begin{equation*}
\begin{split}
&\rho \,\square^L \rho ^{-1}:\Omega^{0,\bullet}(\C^n, \C)
\to\Omega^{0,\bullet}(\C^n,\C)\,,\\
&\rho \,\square^L \rho ^{-1}=\tfrac{1}{2} \cL+ \sum_{j=1}^n 2\pi
d\ov{z}^j \wedge i_{\tfrac{\partial}{\partial\overline{z}_j}}\,,\\
&\rho \,\square^L \rho ^{-1}\vert_{\Omega^{0,0}}=\tfrac{1}{2} \cL\,.
\end{split}
\end{equation*}
The operator $\cL$ is the complex analogue of the harmonic oscillator,
the operators $b$, $b^+$ are creation and annihilation operators
respectively. Each eigenspace of $\cL$ has infinite dimension,
but we can still give an explicit description.
\begin{theorem}[{\cite[Th.\,4.1.20]{MM07}, \cite[Th.\,1.15]{MM08a}}]\label{bkt2.17}
The spectrum of $\cL$ on $L^2(\R^{2n})$ is given by
\begin{equation}\label{bk2.68}
{\spec}(\cL)=\Big\{ 4\pi |\alpha| \,:\, \alpha\in\N ^n\Big\}\,.
\end{equation}
Each $\lambda\in\spec(\cL)$ is an eigenvalue of infinite multiplicity and
an orthogonal basis of the corresponding eigenspace is given by
\begin{equation}\label{bk2.69}
B_\lambda=\Big\{b^{\alpha}\big(z^{\beta} e^{-\pi\sum_i |z_i|^2/2}\big):
\text{$\alpha\in\N^n$ with $4\pi|\alpha|=\lambda$, $\beta\in\N^n$}\Big\}
\end{equation}
where $b^{\alpha}:=b^{\alpha_1}_1\cdots b^{\alpha_n}_n$. Moreover,
$\bigcup\{B_\lambda:\lambda\in\spec(\cL)\}$
forms a complete orthogonal basis of $L^2(\R^{2n})$.
In particular, an orthonormal basis of
$\Ker (\cL)$ is
\begin{equation}\label{bk2.70}
\Big\{\varphi_\beta(z)=\big(\tfrac{\pi ^{|\beta|}}{\beta!}\big)^{1/2}z^\beta
e^{-\pi\sum_i |z_i|^2/2}\,:\beta\in\N^n\Big\}\,.
\end{equation}
\end{theorem}
Let $\cP:L^2 (\R^{2n})\longrightarrow\Ker(\cL)$ be the orthogonal projection
and let $\cP(Z,Z')$ denote its kernel with respect to $dZ'$.
We call $\cP(\cdot,\cdot)$ the Bergman kernel of $\cL$.
Obviously
$\cP(Z,Z')=\sum_{\beta}\varphi_\beta(z)\,\overline{\varphi}_\beta(z')$ so
we infer from \eqref{bk2.70} that
\begin{equation}\label{toe1.3}
\cP(Z,Z^{\prime}) = \exp\big(-\tfrac{\pi}{2}\textstyle\sum_{i=1}^n
\big(|z_i|^2+|z^{\,\prime}_i|^2 -2z_i\overline{z}_i^{\,\prime}\big)\big)\,.
\end{equation}
\subsection{Asymptotic expansion of Bergman kernel}\label{s3.2}
In Sections \ref{s3.2}-\ref{s3.6} we assume that $(X,\omega)$ is a compact K\"ahler manifold and $(L,h^L)$ is a Hermitian holomorphic line bundle satisfying \eqref{prequantum}. For the sake of simplicity, we suppose that Riemannian metric $g^{TX}$ is the metric associated to $\omega$, that is, $g^{TX}(u,v)=\omega(u,Jv)$ (or, equivalently, $\Theta=\omega$).

\noindent
In order to state the result about the asymptotic expansion we start by describing our identifications and notations.

\noindent
\textbf{\emph{Normal coordinates.\/}}
 Let $a^X$ be the injectivity radius of $(X, g^{TX})$.
We denote by $B^{X}(x,\var)$ and  $B^{T_xX}(0,\var)$ the open balls
 in $X$ and $T_x X$ with center $x$ and radius $\var$, respectively.
 Then the exponential map $ T_x X\ni Z \to \exp^X_x(Z)\in X$ is a
 diffeomorphism from $B^{T_xX} (0,\var)$ onto $B^{X} (x,\var)$ for
 $\var\leqslant a^X$.  {}From now on, we identify $B^{T_xX}(0,\var)$
 with $B^{X}(x,\var)$ via the exponential map for $\var \leqslant a^X$.
 Throughout what follows, $\varepsilon$ runs in the
fixed interval $]0, a^X/4[$.


\noindent
\textbf{\emph{Basic trivialization.\/}}
 We fix $x_0\in X$.
 For $Z\in B^{T_{x_0}X}(0,\var)$ we identify
 $(L_Z, h^L_Z)$, $(E_Z, h^E_Z)$ and $(L^p\otimes E)_Z$
 to $(L_{x_0},h^L_{x_0})$, $(E_{x_0},h^E_{x_0})$ and $(L^p\otimes E)_{x_0}$
 by parallel transport with respect to the connections
 $\nabla ^L$, $\nabla ^E$ and $\nabla^{L^p\otimes E}$ along the curve
 \[\gamma_Z :[0,1]\ni u \to \exp^X_{x_0} (uZ)\,.\]
This is the basic trivialization we use in this paper.

Using this trivialization we identify $f\in \cC^\infty(X,\End(E))$ to a family
$\{f_{x_0}\}_{x_0\in X}$ where $f_{x_0}$ is the function $f$ in
normal coordinates near $x_0$, i.e.,
\[
f_{x_0}:B^{T_{x_0}X}(0,\var)\to\End(E_{x_0}),\quad
f_{x_0}(Z)=f\circ\exp^X_{x_0}(Z)\,.
\]
In general, for functions in the normal coordinates,
we will add a subscript $x_0$ to indicate the base point $x_0\in X$.
Similarly,
$P_p(x,x')$ induces in terms of the basic trivialization a smooth section
$(Z,Z')\mapsto P_{p,\,x_0}(Z,Z')$
of $\pi ^* \End(E)$ over $\{(Z,Z')\in TX\times_{X} TX:|Z|,|Z'|<\var\}$,
which depends smoothly on $x_0$. Here we identify a section
$S\in \cC^\infty \big(TX\times_{X}TX,\pi ^* \End (E)\big)$
 with the family $(S_x)_{x\in X}$, where
 $S_x=S|_{\pi^{-1}(x)}$. 

\noindent
\textbf{\emph{Coordinates on $T_{x_0}X$.\/}}
Let us choose
an orthonormal basis $\{ w_i\}_{i=1}^n$ of $T^{(1,0)}_{x_0} X$.
Then $e_{2j-1}=\tfrac{1}{\sqrt{2}}(w_j+\overline{w}_j)$ and
$e_{2j}=\tfrac{\sqrt{-1}}{\sqrt{2}}(w_j-\overline{w}_j)$, $j=1,\dotsc,n\, $
form an orthonormal basis of $T_{x_0}X$.
We use coordinates on $T_{x_0}X\simeq\R^{2n}$
given by the identification
\begin{equation}\label{n11}
\R^{2n}\ni (Z_1,\ldots, Z_{2n}) \longmapsto \sum_i
Z_i e_i\in T_{x_0}X.
\end{equation}
In what follows we also use complex coordinates $z=(z_1,\ldots,z_n)$
on $\C^n\simeq\R^{2n}$.

\noindent
\textbf{\emph{Volume form on $T_{x_0}X$.\/}}
Let us denote by $dv_{TX}$ the Riemannian volume form
on $(T_{x_0}X, g^{T_{x_0}X})$, there exists
a smooth positive function $\kappa_{x_0}:T_{x_0}X\to\R$, satisfying
 \begin{equation} \label{atoe2.7}
 dv_X(Z)= \kappa_{x_0}(Z) dv_{TX}(Z),\quad \kappa_{x_0}(0)=1,
 \end{equation}
 where the subscript $x_0$ of
 $\kappa_{x_0}(Z)$ indicates the base point $x_0\in X$.

\noindent
\textbf{\emph{Sequences of operators.\/}}
Let $\Theta_p: L^2(X,L^p\otimes E)\longrightarrow L^2(X,L^p\otimes E)$
be a sequence of continuous linear  operators with smooth kernel
$\Theta_p(\cdot,\cdot)$ with respect to $dv_X$ (e.g.\,$\Theta_p=T_{f,\,p}$).
 Let $\pi : TX\times_{X} TX \to X$ be the natural projection from the
 fiberwise product of $TX$ on $X$.
In terms of our basic trivialization, $\Theta_p(x,y)$ induces
a family of smooth sections
$Z,Z'\mapsto \Theta_{p,\,x_0}(Z,Z^\prime)$
of $\pi^*\End(E)$ over $\{(Z,Z')\in TX\times_{X} TX:|Z|,|Z'|<\var\}$,
which depends smoothly on $x_0$. 

We denote by $\abs{\Theta_{p,\,x_0}(Z,Z^\prime)}_{\cC^l(X)}$
the $\cC^{l}$ norm with respect to the parameter $x_0\in X$. We say that
\[\Theta_{p,\,x_0}(Z,Z^\prime)=\mO(p^{-\infty})\,,\quad p\to\infty\] if
for any $l,m\in \N$, there exists $C_{l,m}>0$ such that
$\abs{\Theta_{p,\,x_0}(Z,Z^\prime)}_{\cC^{m}(X)}\leqslant C_{l,m}\, p^{-l}$.

The asymptotics of the Bergman kernel will be described in terms of the
Bergman kernel $\cP_{x_0}(\cdot,\cdot)=\cP(\cdot,\cdot)$ of the model operator $\cL$ on $T_{x_0}X\cong \R^{2n}$.
Recall that $\cP(\cdot,\cdot)$ was defined in \eqref{toe1.3}.
\begin{notation}\label{noe2.7}
Fix $k\in\N$ and $\var^\prime\in\,]0,a^X[$\,.
Let \[\{Q_{r,\,x_0}\in \End(E)_{x_0}[Z,Z^{\prime}]\,:\,0\leqslant r\leqslant k,\,x_0\in X\}\]  be a family
of polynomials
in $Z,Z^ \prime$, which is smooth with respect to
the parameter $x_0\in X$. We say that
\begin{equation} \label{toe2.7}
p^{-n} \Theta_{p,\,x_0}(Z,Z^\prime)\cong \sum_{r=0}^k
(Q_{r,\,x_0} \cP_{x_0})(\sqrt{p}Z,\sqrt{p}Z^{\prime})p^{-r/2}
+\mO(p^{-(k+1)/2})\,,
\end{equation}
on $\{(Z,Z^\prime)\in TX\times_X TX:\abs{Z},\abs{Z^{\prime}}<\var^\prime\}$
if there exist $C_0>0$ and a decomposition
\begin{equation} \label{toe2.71}
\begin{split}
p^{-n} \Theta_{p,x_0}(Z,Z^\prime)&-\sum_{r=0}^k
(Q_{r,\,x_0} \cP_{x_0})(\sqrt{p}Z,\sqrt{p}Z^{\prime})\kappa_{x_0}^{-1/2}(Z)\kappa_{x_0}^{-1/2}(Z')p^{-r/2}\\
&=\Psi_{p,\,k,\,x_0}(Z,Z^\prime)+\mO(p^{-\infty})\,,
\end{split}
\end{equation}
where $\Psi_{p,\,k,\,x_0}$ satisfies the following estimate:
for every $l\in\N$ there exist $C_{k,\,l}>0$, $M>0$ such that for all
$p\in\N^{*}$
\begin{equation}
|\Psi_{p,\,k,\,x_0}(Z,Z^\prime)|_{\cC^l(X)}\leqslant  \,C_{k,\,l}\,p^{-(k+1)/2}
(1+\sqrt{p}\,|Z|+\sqrt{p}\,|Z^{\prime}|)^M \,
e^{-C_0\,\sqrt{ p}\,|Z-Z^{\prime}|}\,,
\end{equation}
on
$\{(Z,Z^\prime)\in TX\times_X TX:\abs{Z},\abs{Z^{\prime}}<\var^\prime\}$.
 \end{notation}

\noindent
\textbf{\emph{The sequence $P_p$\,.\/}}
We can now state the asymptotics of the Bergman kernel.
First we observe that the Bergman kernel
decays very fast outside the diagonal of $X\times X$.

Let ${\mathbf{f}} : \R \to [0,1]$ be a smooth even function such that
${\mathbf{f}}(v)=1$ for $|v| \leqslant  \var/2$, 
and ${\mathbf{f}}(v) = 0$ for $|v| \geqslant \var$. Set
\begin{equation} \label{0c3}
F(a)= \Big(\int_{-\infty}^{+\infty}{\mathbf{f}}(v) dv\Big)^{-1} 
\int_{-\infty}^{+\infty} e ^{i v a}\, {\mathbf{f}}(v) dv.
\end{equation}
Then $F(a)$ is an even function and lies in the Schwartz space 
$\mathcal{S} (\R)$ and $F(0)=1$.

\noindent
We have the \emph{far
off-diagonal} behavior of the Bergman kernel: 
\begin{theorem}[{\cite[Prop.\,4.1]{DLM04a}}]\label{tue16}
For any $l,m\in\N$ and $\var>0$, there exists a positive constant $C_{l,m,\var}>0$ 
such that for any $p\geqslant 1$, $x,x'\in X$, the following estimate holds:
\begin{equation}\label{0c7}
\left|F(D_p)(x,x') 
- P_{p}(x,x')\right|_{\cC ^m(X\times X)}
\leqslant C_{l,m,\var} p^{-l}.
\end{equation}
Especially, 
\begin{equation}\label{toe2.6a}
|P_p(x,x')|_{\cC^m(X\times X)} \leqslant C_{l,m,\var}\, p^{-l}\,,\quad \text{on $\{(x,x')
\in X\times X: d(x,x') \geqslant \var\}$} \,.
\end{equation}
The $\cC ^m$ norm in \eqref{0c7} and \eqref{toe2.6a} is induced by
$\nabla^L$, $\nabla^E$, $h^L$, $h^E$ and $g^{TX}$.
\end{theorem}
Next we formulate the \emph{near off-diagonal} expansion of the Bergman kernel.
\begin{theorem}[{\cite[Th.\,\,4.18$^\prime$]{DLM04a}}]\label{tue17}
There exist polynomials $J_{r,\,x_{0}}\in \End(E)_{x_0}[Z,Z']$
in $Z,Z'$ with the same parity as $r$, such that
for any $k\in \N$, $\varepsilon\in]0, a^X/4[$\,,
we have
\begin{equation} \label{toe2.9}
p^{-n} P_{p,\,x_0}(Z,Z^\prime)\cong \sum_{r=0}^k
(J_{r,\,x_0} \cP_{x_0})(\sqrt{p}Z,\sqrt{p}Z^{\prime})p^{-\frac{r}{2}}
+\mO(p^{-\frac{k+1}{2}})\,,
\end{equation}
on the set $\{(Z,Z^\prime)\in TX\times_X TX:\abs{Z},\abs{Z^{\prime}}<2\var\}$,
in the sense of Notation \ref{noe2.7}.
\end{theorem}
Let us briefly explain the idea of the proof for Theorems \ref{tue16}\,--\,\ref{tue17}.   Using the spectral gap property from Theorem \ref{bkt1.1}, 
we obtain \eqref{0c7}. By  finite propagation speed of solutions of hyperbolic equations,
we obtain that $F(D_p)(x,x')=0$ if $d(x,x') \geqslant \var$ and $F(D_p)(x,\cdot)$ depends only on the restriction $D_p|_{B(x,\varepsilon)}$, so \eqref{toe2.6a} follows. This shows that we can localize the asymptotics
of $P_p(x_0,x')$ in the neighborhood of $x_0$.
By pulling back all our objects by the exponential map to the tangential space and suitably extending them we can work on $\R^{2n}$. Thus we can use the explicit description of the Bergman kernel of the model operator $\cL$ given in Section \ref{toes1}.
To conclude the proof, we combine the spectral gap property,
the rescaling of the coordinates and functional analytic
techniques inspired by Bismut-Lebeau \cite[\S 11]{BL91}.

By setting $\bb_r(x_0)=(J_{2r,\,x_0}\cP_{x_0})(0,0)$, we get
from \eqref{toe2.9} the following diagonal expansion of the Bergman kernel.
\begin{corollary}\label{bkt2.18}
For any $k,l\in \N$, there exists $C_{k,l}>0$
such that for any $p\in \N$,
\begin{equation}\label{bk2.6}
\Big |P_p(x,x)- \sum_{r=0}^{k} \bb_r(x) p^{n-r}
\Big |_{\cC^l(X)} \leqslant C_{k,\,l}\, p^{n-k-1}\,,
\end{equation}
where $\bb_0(x)=\Id_E$.
\end{corollary}
The existence of the expansion \eqref{bk2.6} and the form of the leading term
was proved by \cite{T90,Catlin99,Ze98}.

The calculation of the coefficients $\bb_r$ is of great importance.
For this we need $J_{r,\,x_{0}}$, which are obtained by computing
the operators $\cF_{r,\,x_{0}}$ defined by the smooth kernels
\begin{align}\label{bk2.24}
    \cF_{r,\,x_{0}}(Z,Z')= J_{r,\,x_{0}}(Z,Z')\cP(Z,Z')
\end{align}
with respect to $dZ'$. Our strategy (already used in \cite{MM07,MM08a})
is to rescale the Kodaira-Laplace operator, take the Taylor expansion
of the rescaled operator and apply resolvent analysis.

\noindent
\textbf{\emph{Rescaling $\Box_p$ and Taylor expansion.\/}}
For  $s \in \cC^{\infty}(\R^{2n}, E_{x_0})$, $Z\in \R^{2n}$,
 $|Z|\leq 2\var$, and
for $t=\frac{1}{\sqrt{p}}$, set
\begin{align}\label{bk2.21}
\begin{split}
&(S_{t} s ) (Z) :=s (Z/t),  \\
&  \cL_{t}:= S_t^{-1} \kappa^{1/2}\, t^2 (2\, \Box_{p})\kappa^{-1/2} S_t\,.
\end{split}\end{align}
Then by \cite[Th.\,4.1.7]{MM07},  there exist
second order differential operators $\mO_{r}$
such that we have an asymptotic expansion in $t$ when $t\to 0$,
\begin{align}\label{bk2.22}
    \cL_{t} = \cL_{0} + \sum_{r=1}^{m} t^r \mO_{r} + \cO(t^{m+1}).
\end{align}
\noindent
From \cite[Th.\,\,4.1.21,\,4.1.25]{MM07}, we obtain
\begin{align}\label{bk2.30}
    \cL_0=& \sum_j b_jb^+_j=\cL, \qquad \mO_1=0.
\end{align}

\noindent
\textbf{\emph{Resolvent analysis.\/}}
We define by recurrence
$f_r(\lambda)\in \End(L^2(\R^{2n}, E_{x_0}))$ by
\begin{align}\label{bk2.23}
f_{0}(\lambda) = (\lambda -\cL_0)^{-1},
\quad f_r(\lambda)=(\lambda -\cL_0)^{-1}
 \sum_{j=1}^{r} \mO_j   f_{r-j}(\lambda).
\end{align}
Let $\delta$ be the counterclockwise
oriented circle in $\C$ of center $0$ and radius $\pi/2$.
Then by \cite[(1.110)]{MM08a} (cf. also \cite[(4.1.91)]{MM07})
\begin{equation}\label{bk2.77}
\cF_{r,\,x_{0}}= \frac{1}{2\pi \sqrt{-1}} \int_{\delta}  f_r (\lambda)d \lambda.
\end{equation}

Since the spectrum of $\cL$ is well understood we can calculate
the coefficients $\cF_{r,\,x_{0}}$. Set $\cP^\bot= \Id - \cP$.
From Theorem \ref{bkt2.17}, \eqref{bk2.30} and \eqref{bk2.77}, we get
\begin{align}\label{bk2.31}
    \begin{split}
	\cF_{0,\,x_{0}}=& \cP,  \quad \cF_{1,\,x_{0}}= 0,\\
\cF_{2,\,x_{0}}=&- \cL^{-1} \cP^\bot\mO_2 \cP
- \cP \mO_2\cL^{-1}\cP^\bot,\\
\cF_{3,\,x_{0}}=&- \cL^{-1} \cP^\bot\mO_3\cP
- \cP \mO_3\cL^{-1}\cP^\bot,
\end{split}\end{align}
and
\begin{align}
\label{bk2.32}
    \begin{split}   \cF_{4,\,x_{0}}&=\cL^{-1}\cP^\bot \mO_{2}\cL^{-1}\cP^\bot \mO_{2} \cP
- \cL^{-1}\cP^\bot \mO_{4} \cP\\
&\quad+ \cP\mO_{2}\cL^{-1}\cP^\bot \mO_{2}\cL^{-1}\cP^\bot
- \cP \mO_{4}\cL^{-1}\cP^\bot \\
&\quad+ \cL^{-1}\cP^\bot \mO_{2} \cP \mO_{2} \cL^{-1}\cP^\bot
-  \cP \mO_{2} \cL^{-2} \cP^\bot  \mO_{2} \cP\\
&\quad- \cP\mO_{2} \cP \mO_{2} \cL^{-2}\cP^\bot
- \cP^\bot \cL^{-2}\mO_{2}  \cP \mO_{2} \cP.
\end{split}\end{align}
In particular, the first two identities of \eqref{bk2.31} imply
\begin{align}\label{bk2.33}
J_{0,\,x_{0}}= 1, \quad J_{1,\,x_{0}}=0.
\end{align}


In order to formulate the formulas for $\bb_1$ and $\bb_2$ we introduce now more notations.
Let $\nabla^{TX}$ be the Levi-Civita connection on $(X, g^{TX})$. We denote by
$R^{TX}=(\nabla^{TX})^2$ the curvature, by $\Ric$ the Ricci curvature and 
by $r^X$ the scalar curvature of $\nabla^{TX}$. 

We still denote by $\nabla ^{E}$  the connection on
 $\End (E)$ induced by $\nabla ^E$.
Consider 
the (positive) Laplacian $\Delta$ 
acting on the functions on $(X, g^{TX})$ and the 
Bochner Laplacian $\Delta^{E}$
on $\cC^{\infty}(X, E)$ and on $\cC^{\infty}(X, \End(E))$. 
Let $\{e_k\}$ be a (local) orthonormal frame of $(TX, g^{TX})$. Then
\begin{align} 
    \Delta^{E} = - \sum_k (\nabla^{E}_{e_{k}}\nabla^{E}_{e_{k}}
- \nabla^{E}_{\nabla^{TX}_{e_{k}}e_{k}}).
\end{align}

Let $\Omega^{q,\,r}(X, \End(E))$ be the space of $(q,r)$-forms on $X$ 
with values in $\End(E)$, and let
\begin{align} \label{abk2.5}
\nabla^{1,0}:\Omega^{q,\bullet}(X, \End(E))\to \Omega^{q+1,\bullet}(X,\End(E))
\end{align}
 be the $(1,0)$-component of the connection $\nabla^E$.
 Let $(\nabla^{E})^*$, $\nabla^{1,0*}, \ov{\partial}^{E*}$
  be the adjoints of $\nabla^{E}$,  $\nabla^{1,0}, \ov{\partial}^{E}$,
 respectively.
Let $D^{1,0}, D^{0,1}$ be the $(1,0)$ and $(0,1)$ components of the
 connection
$\nabla^{T^{*}X}: \cC^\infty(X,T^* X)\to \cC^\infty(X,T^* X\otimes T^*X)$
induced by $\nabla^{TX}$.
In the following, we denote by 
\[
\langle\cdot\,,\cdot \rangle_{\om}:
\Omega^{\bullet,\,\bullet}(X,\End(E))\times\Omega^{\bullet,\,\bullet}(X,\End(E))
\to\cC^\infty(X,\End(E))
\]
the $\C$-bilinear pairing
$\langle\alpha\otimes f,\beta\otimes g \rangle_{\om}
=\langle\alpha,\beta\rangle f\cdot g$, 
for forms $\alpha,\beta\in\Omega^{\bullet,\,\bullet}(X)$ and sections 
$f,g\in\cC^\infty(X,\End(E))$. Put 
\begin{align} \label{bk2.5}
\begin{split}
&R^E_{\Lambda}
=\left \langle R^E, \om\right \rangle_{\om}\, .
\end{split}\end{align}
Let $\Ric_\om= \Ric(J\cdot,\cdot)$
be the $(1,1)$-form associated to $\Ric$.
Set 
\[
 |\Ric_\om|^2 = \sum_{i<j}\Ric_\om(e_{i},e_{j})^{2}\,,
 \quad |R^{TX}|^2 = \sum_{i<j}\sum_{k<l} \langle 
R^{TX}(e_{i},e_{j})e_{k},e_{l}\rangle ^{2},
\]

\begin{theorem}\label{toet4.5}
We have
     \begin{equation}\label{toe4.131}
\bb_{1} = \frac{1}{8\pi}r^X + \frac{\imat}{2\pi}R^E_{\Lambda}\,,
\end{equation}

\begin{equation}\label{toe4.132}
\begin{split}
\pi^2 \bb_{2}&=  - \frac{\Delta r^X}{48}+ \frac{1}{96}|R^{TX}|^2
- \frac{1}{24} |\Ric_\omega|^2 +  \frac{1}{128} (r^X)^2\\
&+\frac{\sqrt{-1}}{32} \Big(2 r^X R^E_{\Lambda} 
- 4 \langle\Ric_\om, R^ E\rangle_{\omega} + \Delta^E  R^E_{\Lambda}\Big) \\
&- \frac{1}{8} (R^E_{\Lambda})^2 + 
 \frac{1}{8}\langle R^ E, R^ E\rangle_{\om}
+ \frac{3}{16} \ov{\partial}^{E*}\nabla^{1,0*} R^ E\,.
\end{split}
\end{equation}
\end{theorem}
The terms $\bb_{1}$, $\bb_{2}$ were computed by
Lu \cite{Lu00} (for $E=\C$, the trivial line bundle with trivial
metric),  X.~Wang \cite{Wang05},
L.~Wang \cite{Wangl03}, in various degree of generality.
The method of these authors is to construct appropriate peak sections
as in \cite{T90}, using
H{\"o}rmander's $L^2$ $\overline\partial$-method.
In \cite[\S 5.1]{DLM04a}, Dai-Liu-Ma computed  $\bb_{1}$
by using the heat  kernel, and in \cite[\S 2]{MM08a}, \cite[\S 2]{MM06}
(cf.\ also \cite[\S 4.1.8, \S 8.3.4]{MM07}), we computed $\bb_{1}$ in the symplectic case.
A new method for calculating $\bb_2$ was given in \cite{MM10}.
\subsection{Asymptotic expansion of Toeplitz operators} \label{s3.3}
We stick to the situation studied in the previous Section, namely, $(X,\omega)$ is a compact K\"ahler manifold and $(L,h^L)$ is a Hermitian holomorphic line bundle satisfying \eqref{prequantum}, and $g^{TX}$ is the Riemannian metric associated to $\omega$.

\noindent
In order to develop the calculus of Toeplitz kernels we use the Bergman kernel expansion
\eqref{toe2.9} and the Taylor expansion of the symbol. We are thus led to a kernel calculus on $\C^n$ with kernels of the form $F\cP$, where $F$ is a polynomial.
This calculus can be completely described in terms of the spectral decomposition \eqref{bk2.68}-\eqref{bk2.69} of the model operator $\cL$.

For a polynomial $F$ in $Z,Z^\prime$, we denote by $F\cP$
the operator on $L^2(\R^{2n})$ defined by the kernel
$F(Z,Z^\prime)\cP(Z,Z^\prime)$ and the volume form $dZ$
according to \eqref{lm2.01c}.

The following very useful Lemma \cite[Lemma 7.1.1]{MM07}
describes the calculus of the kernels
$(F\cP)(Z,Z^\prime):=F(Z,Z^\prime)\cP(Z,Z^\prime)$.
\begin{lemma}\label{toet1.1}
For any
$F,G\in\C[Z, Z^{\prime}]$ there exists a polynomial
$\cK[F,G]\in\C[Z, Z^{\prime}]$
with degree $\deg\cK[F,G]$ of the same parity as
$\deg F+\deg G$, such that
\begin{equation}\label{toe1.6}
((F\cP) \circ (G\cP))(Z, Z^{\prime}) =
\cK[F,G](Z, Z^{\prime}) \cP( Z, Z^{\prime}).
\end{equation}
\end{lemma}
Let us illustrate how Lemma \ref{toet1.1} works. First observe
that from \eqref{bk2.67} and \eqref{toe1.3},
for any polynomial $g(z,\ov{z})\in\C[z,\ov{z}]$, we get
\begin{align}\label{toe1.7}
\begin{split}
&b_{j\,,z}\,\cP(Z, Z^{\prime}) = 2\pi (\ov{z}_j- \ov{z}_j^{\,\prime})
\cP(Z, Z^{\prime}),\\
&[g(z,\ov{z}),b_{j\,,z}]=2\frac{\partial}{\partial z_j}g(z,\ov{z})\,.
\end{split}\end{align}
Now \eqref{toe1.7} entails
\begin{equation}\label{toe1.7a1}
\ov{z}_j\,\cP(Z, Z^{\prime})=\frac{b_{j\,,z}}{2\pi}\,
\cP(Z, Z^{\prime})+\ov{z}_j^{\,\prime}\cP(Z, Z^{\prime}).
\end{equation}
Specializing \eqref{toe1.7} for $g(z,\overline{z})=z_i$ we get
\begin{equation}\label{toe1.7b1}
z_i\,b_{j\,,z}\cP(Z, Z^{\prime})
=b_{j\,,z}(z_i\cP)(Z, Z^{\prime})+2\delta_{ij}\cP(Z, Z^{\prime}),
\end{equation}
Formulas \eqref{toe1.7a1} and \eqref{toe1.7b1} give
\begin{align} \label{toe1.11}
\begin{split}
z_i\ov{z}_j\,\cP(Z, Z^{\prime})
&=\frac{1}{2\pi}\,b_{j\,,z}\,z_i\cP(Z, Z^{\prime})
+\frac{1}{\pi}\, \delta_{ij}
\cP(Z, Z^{\prime})+z_i\ov{z}^{\,\prime}_j\,\cP(Z, Z^{\prime})\,.
\end{split}\end{align}

Using the preceding formula we calculate further some examples for
the expression $\cK[F,G]$ introduced \eqref{toe1.6}.
We use the spectral decomposition of $\cL$ in the following way.
If $\varphi(Z)= b^\alpha z^\beta
\exp\big(-\frac{\pi}{2} \sum_{j=1}^n |z_j|^2\big)$
with $\alpha,\beta\in \N^n$, then Theorem \ref{bkt2.17}
implies immediately that
 \begin{align}\label{toe1.4}
(\cP \varphi)(Z)= \left \{  \begin{array}{ll}
\displaystyle{z^\beta \exp\Big (-\frac{\pi}{2} \sum_{j=1}^n |z_j|^2\Big ) }
& \mbox{if} \,\, |\alpha|=0,  \\
 0& \mbox{if} \,\, |\alpha|>0.\end{array}\right.
\end{align}
The identities 
\eqref{toe1.7a1}, \eqref{toe1.11} and \eqref{toe1.4} imply that
\begin{align} \label{toe1.12}
\begin{split}
&\cK[1,\ov{z}_j]\cP =\cP\circ(\ov{z}_j\cP)
=\ov{z}^{\,\prime}_j\cP,\quad
\cK[1,z_j]\cP= \cP\circ(z_j\cP)= z_j\cP,\\
&\cK[z_i,\ov{z}_j]\cP =(z_i \cP)\circ(\ov{z}_j\cP)
= z_i  \cP\circ(\ov{z}_j\cP)=z_i\ov{z}^{\,\prime}_j\cP,\\
&\cK[\ov{z}_i,z_j]\cP= (\ov{z}_i\cP)\circ(z_j\cP)
= \ov{z}_i\cP\circ(z_j\cP) = \ov{z}_i z_j\cP,\\
&\cK[z_i^{\,\prime},\ov{z}_j]\cP =(z_i^{\,\prime} \cP)\circ(\ov{z}_j\cP)
=  \cP\circ(z_i\ov{z}_j\cP)= \tfrac{1}{\pi}\delta_{ij} \cP
+z_i \ov{z}_j^{\,\prime}\cP,\\
&\cK[\ov{z}_i^{\,\prime},z_j]\cP= (\ov{z}_i^{\,\prime}\cP)\circ(z_j\cP)
= \cP\circ(\ov{z}_i z_j\cP) = \tfrac{1}{\pi}\delta_{ij} \cP
+\ov{z}_i^{\,\prime} z_j\cP.
\end{split}\end{align}
Thus we get
\begin{align} \label{toe1.13}
\begin{split}
& \cK[1,\ov{z}_j]= \ov{z}^{\,\prime}_j,\quad \cK[1,z_j]=z_j,\\
&\cK[z_i,\ov{z}_j]=z_i\ov{z}^{\,\prime}_j,\quad \cK[\ov{z}_i,z_j]= \ov{z}_i z_j,\\
&\cK[\ov{z}_i^{\,\prime},z_j]= \cK[z_j^\prime,\ov{z}_i]
= \tfrac{1}{\pi}\delta_{ij} +\ov{z}_i^{\,\prime} z_j.
\end{split}\end{align}

To simplify our calculations, we introduce the following notation.
For any polynomial $F\in\C[Z,Z^{\prime}]$ we denote by $(F\cP)_p$
the operator defined by the kernel
$p^{n}(F\cP)(\sqrt{p}Z, \sqrt{p}Z^{\prime})$, that is,
\begin{equation}
((F\cP)_p \varphi)(Z)=\int_{\R^{2n}}p^{n}
(F\cP)(\sqrt{p}Z, \sqrt{p}Z^{\prime})\varphi(Z^{\prime})\,dZ^{\prime}\,,\quad
\text{for $\varphi\in L^2(\R^{2n})$.}
\end{equation}
Let $F,G\in \C[Z,Z^\prime]$. By a change of variables we obtain
\begin{align} \label{toe1.15}
((F\cP)_p\circ (G\cP)_p)(Z,Z^\prime)
= p^{n}((F\cP)\circ (G\cP))(\sqrt{p}Z, \sqrt{p}Z^{\prime}).
\end{align}

\noindent
We examine now the asymptotic expansion of the kernel of the Toeplitz 
operators $T_{f,\,p}$.
The first observation is that outside the diagonal of $X\times X$,
the kernel of $T_{f,\,p}$ has the growth $\cO(p^{-\infty})$, as $p\to\infty$. 
\begin{lemma}[{\cite[Lemma\,4.2]{MM08b}}] \label{toet2.1}
For every $\varepsilon>0$ and every $l,m\in\N$, 
there exists $C_{l,m,\varepsilon}>0$ such that 
\begin{equation} \label{toe2.6b}
|T_{f,\,p}(x,x')|_{\cC^m(X\times X)}\leqslant C_{l,m,\varepsilon}p^{-l}
\end{equation}
for all $p\geqslant 1$ and all $(x,x')\in X\times X$ 
with $d(x,x')>\varepsilon$,
where the $\cC^m$-norm is induced by $\nabla^L,\nabla^E$ and $h^L,h^E,g^{TX}$.
\end{lemma} 
\begin{proof}
Due to \eqref{toe2.6a},  \eqref{toe2.6b} holds 
if we replace $T_{f,\,p}$ by $P_p$. 
Moreover, from \eqref{toe2.9}, for any $m\in \N$, 
there exist $C_m>0, M_m>0$ such that 
$|P_p(x,x')|_{\cC^m(X\times X)}<Cp^{M_m}$ for all $(x,x')\in X\times X$.
These two facts and formula \eqref{toe2.5} imply the Lemma. 
\end{proof}
The near off-diagonal expansion of the Bergman kernel \eqref{toe2.9} and the kernel calculus on $\C^n$ presented above imply the near off-diagonal expansion of the Toeplitz kernels.
(cf.\,\cite[Lemma\,4.6]{MM08b}, \cite[Lemma\,7.2.4]{MM07})
\begin{theorem} \label{toet2.3}
Let $f\in\cC^\infty(X,\End(E))$.
There exists a family \[\{Q_{r,\,x_0}(f)\in\End(E)_{x_0}[Z,Z^{\prime}]:r\in\N,\,x_0\in X\}\,,\]
depending smoothly on the parameter $x_0\in X$, where 
$Q_{r,\,x_0}(f)$
are polynomials with the same parity as $r$ and such that for every $k\in \N$,
$\varepsilon\in]0, a^X/4[$\,,
\begin{equation} \label{toe2.13}
p^{-n}T_{f,\,p,\,x_0}(Z,Z^{\prime})
\cong \sum^k_{r=0}(Q_{r,\,x_0}(f)\cP_{x_0})(\sqrt{p}Z,\sqrt{p}Z^{\prime})
p^{-r/2} + \mO(p^{-(k+1)/2})\,,
\end{equation}
on the set $\{(Z,Z^\prime)\in TX\times_X TX:\abs{Z},\abs{Z^{\prime}}<2\var\}$,
in the sense of Notation \ref{noe2.7}. Moreover,
  $Q_{r,\,x_0}(f)$ are expressed by
  \begin{equation} \label{toe2.14}
  Q_{r,\,x_0}(f) = \sum_{r_1+r_2+|\alpha|=r}
    \cK\Big[J_{r_1,\,x_0}\;,\;
  \frac{\partial ^\alpha f_{\,x_0}}{\partial Z^\alpha}(0)
  \frac{Z^\alpha}{\alpha !} J_{r_2,\,x_0}\Big]\,.
  \end{equation}
where $\cK[\cdot,\cdot]$ was introduced in \eqref{toe1.6}. We have,
  \begin{align} \label{toe2.15}
  Q_{0,\,x_0}(f)= f(x_0)\in\End(E_{x_0}) .
  \end{align}
  \end{theorem}


\begin{proof} Estimates \eqref{toe2.5} and \eqref{toe2.6b} learn that for
$\abs{Z},\abs{Z^{\prime}}<\var/2$, $T_{f,\,p,\,x_0}(Z,Z^\prime)$
is determined up to terms of order $\cO(p^{-\infty})$
 by the behavior of $f$ in $B^X(x_0, \varepsilon)$.
Let $\rho: \R\to [0,1]$ be a smooth even function such that
\begin{equation}\label{alm4.19}
\rho (v)=1  \  \  {\rm if} \  \  |v|<2;
\quad \rho (v)=0 \   \   {\rm if} \  |v|>4.
\end{equation}
For $\abs{Z},\abs{Z^{\prime}}<\var/2$, we get
\begin{equation}\label{toe2.17}
\begin{split}
&T_{f,\,p,\,x_0}(Z,Z^{\prime})=\cO(p^{-\infty})\\ &+\int_{{ T_{x_0}X}}
P_{p,x_0}(Z,Z^{\prime\prime})
\rho(2|Z^{\prime\prime}|/\var)
f_{x_0}(Z^{\prime\prime})P_{p,x_0}(Z^{\prime\prime},Z^{\prime})
 \kappa_{x_0}(Z^{\prime\prime})
\,dv_{TX}(Z^{\prime\prime})\,.
\end{split}
\end{equation}
We consider the Taylor expansion of $f_{x_0}$:
\begin{equation}\label{toe2.18}
\begin{split}
f_{x_0}(Z)&=\sum_{|\alpha|\leqslant k}
\frac{\partial^\alpha f_{x_0}}{\partial Z^\alpha}(0)
\frac{ Z^\alpha}{\alpha!}+\cO(|Z|^{k+1})\\
&=\sum_{|\alpha|\leqslant k} p^{-|\alpha|/2}
\frac{\partial^\alpha f_{x_0}}{\partial Z^\alpha}(0)
\frac{ (\sqrt{p}Z)^\alpha}{\alpha!}
+p^{-\frac{k+1}{2}}\cO(|\sqrt{p}Z|^{k+1}).
\end{split}
\end{equation}
We multiply now the expansions given in \eqref{toe2.18} and \eqref{toe2.9}. Note the presence of $\kappa_{x_0}$ in the definition \eqref{toe2.71} of \eqref{toe2.7}. Hence we obtain the expansion of
\begin{equation*}
\kappa_{x_0}^{1/2}(Z)P_{p,\,x_0}(Z,Z^{\prime\prime})
(\kappa_{x_0}f_{x_0})(Z^{\prime\prime})
P_{p,\,x_0}(Z^{\prime\prime},Z^{\prime})\kappa_{x_0}^{1/2}(Z^{\prime})
\end{equation*}
which we substitute in \eqref{toe2.17}. We integrate then on
$T_{x_0}X$ by using the change of variable $\sqrt{p}\,Z^{\prime\prime}=W$
and conclude \eqref{toe2.13} and \eqref{toe2.14}
by using formulas \eqref{toe1.6} and \eqref{toe1.15}.

\noindent
From \eqref{bk2.33} and \eqref{toe2.14}, we get
\begin{align} \label{toe2.19}
Q_{0,\,x_0}(f)= \cK[1, f_{x_0}(0)] = f_{x_0}(0) = f(x_0)\,.
\end{align}
The proof of Lemma \ref{toet2.3} is complete.
\end{proof}
As an example, we compute $Q_{1,\,x_0}(f)$. By \eqref{bk2.22}, \eqref{toe1.13} and \eqref{toe2.14} we obtain
\begin{equation} \label{toe2.20}
\begin{split}
Q_{1,\,x_0}(f)= \cK\Big[1, \frac{\partial f_{x_0}}{\partial Z_j}(0) Z_j\Big]
=\frac{\partial f_{x_0}}{\partial z_j}(0) z_j+\frac{\partial f_{x_0}}{\partial\ov{z}_j}(0)\ov{z}^{\,\prime}_j\,.
\end{split}
\end{equation}

\begin{corollary}\label{toec2.1}
For any $f\in\cC^\infty(X,\End(E))$, we have
\begin{equation}\label{bk4.2}
T_{f,\,p}(x,x)= \sum_{r=0}^{\infty} \bb_{r,f}(x) p^{n-r}+\cO(p^{-\infty})\,,
\quad \bb_{r,f}\in\cC^\infty(X,\End(E))\,.
\end{equation}
\end{corollary}
\begin{proof}
By taking $Z=Z'=0$ in \eqref{toe2.13} we obtain \eqref{bk4.2}, with $\bb_{r,f}(x)=Q_{2r,x}(f)$. 
\end{proof}
Since we have the precise formula \eqref{toe2.13} for $Q_{2r,x}(f)$ we can give a closed formula for the first coefficients $\bb_{r,f}$.  
In \cite{MM10}, we computed the coefficients
$\bb_{1,f}, \bb_{2,f}$,
from \eqref{bk4.2}. These computations are also 
relevant in K{\"a}hler geometry 
(cf.\ \cite{Fine08}, \cite{Fine10}, \cite{LM07}).
\begin{theorem}[{\cite[Th.\,0.1]{MM10}}] \label{toet4.6} 

For any $f\in\cC^\infty(X,\End(E))$, we have{\rm:}
\begin{equation}\label{bk4.3}
\bb_{0,\,f}=f, \quad \bb_{1,\,f} = \frac{r^X}{8\pi}  f
+ \frac{\sqrt{-1}}{4\pi}
\left(R^E_{\Lambda}  f + fR^E_{\Lambda} \right)
-  \frac{1}{4\pi} \Delta^{E} f\,.
\end{equation}
If $f\in \cC^\infty(X)$, then
\begin{equation}\label{abk4.4}
\begin{split}
\pi^2 \bb_{2,\,f}  =& \,\pi^2  \bb_{2}  f
+ \frac{1}{32}\Delta^2 f
- \frac{1}{32}  r^X \Delta f
- \frac{\sqrt{-1}}{8} \big\langle \Ric_\om, 
\partial\ov{\partial}f\big\rangle\\
&+ \frac{\sqrt{-1}}{24} \big\langle df, 
\nabla^E R^E_{\Lambda}\big\rangle_{\om}
 + \frac{1}{24} \big\langle \partial f, \nabla^{1,0*} R^E\big\rangle_{\om}
 - \frac{1}{24} \big\langle \ov{\partial} f, 
\ov{\partial}^{E*} R^E\big\rangle_{\om}\\
&-  \frac{\sqrt{-1}}{8} (\Delta f)R^E_{\Lambda}
+ \frac{1}{4}\big\langle \partial \ov{\partial} f, R^E\big\rangle_{\om}\,.
\end{split}
\end{equation}
\end{theorem}
\subsection{Algebra of Toeplitz operators, Berezin-Toeplitz star-product}\label{s3.6}
Lemma \ref{toet2.1} and Theorem \ref{toet2.3} provide the asymptotic expansion of the kernel of a Toeplitz operator $T_{f,\,p}$\,. Using this lemma we can for example easily obtain the expansion of the kernel of the composition $T_{f,\,p}T_{g,\,p}$, for two sections $f,g\in\cC^{\infty}(X,\End(E))$. The result will be an asymptotic expansion of the type \eqref{toe2.13}. Luckily we can show that the existence of a such asymptotic expansion \emph{characterizes}
Toeplitz operators (in the sense of Definition \ref{toe-def}).
We have the following useful criterion which ensures that
a given family is a Toeplitz operator.
\begin{theorem}\label{toet3.1}
Let $\{T_p:L^2(X,L^p\otimes E)\longrightarrow L^2(X,L^p\otimes E)\}$
be a family of bounded linear operators. Then $\{T_p\}$ is a Toeplitz operator if and only if  satisfies
the following three conditions:
\\[2pt]
(i) For any $p\in \N$,  $P_p\,T_p\,P_p=T_p$\,.
\\[2pt]
(ii) For any $\varepsilon_0>0$ and any $l\in\N$,
there exists $C_{l,\,\varepsilon_0}>0$ such that
for all $p\geqslant 1$ and all $(x,x')\in X\times X$
with $d(x,x')>\varepsilon_0$,
\begin{equation} \label{toe3.1}
|T_{p}(x,x')|\leqslant C_{l,\varepsilon_0}p^{-l}.
\end{equation}
\\[2pt]
(iii) There exists a family of polynomials
$\{\mQ_{r,\,x_0}\in\End(E)_{x_0}
[Z,Z^{\prime}]\}_{x_0\in X}$
such that:
\begin{itemize}
\item[(a)] each $\mQ_{r,\,x_0}$ has the same parity as $r$,
\item[(b)] the family is smooth in $x_0\in X$ and
\item[(c)] there exists $0<\var^\prime<a^X/4$ such that for every  $x_0\in X$,
 every $Z,Z^\prime \in T_{x_0}X$ with  $\abs{Z},\abs{Z^{\prime}}<\var^\prime$
and every $k\in\N$ we have
\begin{equation} \label{toe3.2}
p^{-n}T_{p,\,x_0}(Z,Z^{\prime})\cong
\sum^k_{r=0}(\mQ_{r,\,x_0}\cP_{x_0})
(\sqrt{p}Z,\sqrt{p}Z^{\prime})p^{-\frac{r}{2}} + \mO(p^{-\frac{k+1}{2}}),
\end{equation}
in the sense of Notation \ref{noe2.7}.
\end{itemize}
\end{theorem}
\begin{proof} In view of Lemma \ref{toet2.1} and Theorem \ref{toet2.3} it is easy to see that conditions (i)-(iii) are necessary. To prove the sufficiency we use the following strategy. 
We define inductively the sequence
$(g_l)_{l\geqslant0}$, $g_l\in \cC^\infty(X, \End(E))$ such that
\begin{equation}\label{toe3.4}
T_{p} = \sum_{l=0}^m P_{p}\, g_l \, p^{-l}\, P_{p} +\mO(p^{-m-1})\,,\quad\text{for every $m\geqslant 0$}\,.
\end{equation}
Let us start with the case $m=0$ of \eqref{toe3.4}. For an arbitrary but fixed $x_0\in X$, we set 
\begin{equation}\label{toe3.5}
g_0(x_0)=\mQ_{0,\,x_0}(0,0)\,\in\End(E_{x_0})\,.
\end{equation}
Then show that
\begin{equation}\label{toe3.60}
p^{-n} (T_{p} -  T_{g_0,\,p})_{x_0}(Z,Z^\prime)\cong \mO(p^{-1})\,,
\end{equation}
which implies the case $m=0$ of \eqref{toe3.4}, namely,
\begin{equation}\label{toe3.6}
T_{p}=P_p\,g_0\,P_p + \mO(p^{-1}).
\end{equation}
A crucial point here is the following result. 
\begin{proposition}[{\cite[Prop.\,4.11]{MM08b}}]\label{toet3.2}
In the conditions of Theorem \ref{toet3.1} we have
\[\mQ_{0,\,x_0}(Z,Z^{\prime})=\mQ_{0,\,x_0}(0,0)\]
for all $x_0\in X$ and all $Z,Z^{\prime}\in T_{x_0}X$.
\end{proposition}
The proof is quite technical, so we refer to \cite[p.\,585-90]{MM08b} for all the details.
The idea is the following. We observe first that $\mQ_{0,\,x_0}\in\End(E_{x_0}) \circ I_{\C\otimes E} [Z,Z^{\prime}]$,
and $\mQ_{0,\,x_0}$ is a polynomial in $z,\ov{z}^\prime$ (cf.\ \cite[Lemma\,4.12]{MM08b}).
For simplicity we denote 
$F_{x}=\mQ_{0,\,x}|_{\C\otimes E}\in \End(E_{x})$. 
Let $F_{x}=\sum_{i\geqslant0}F^{(i)}_{x}$ be the decomposition of $F_{x}$
in homogeneous polynomials $F^{(i)}_{x}$ of degree $i$. 
We show that $F^{(i)}_{x}$ vanish identically for $i>0$, that is,
\[
F^{(i)}_{x}(z,\ov{z}^{\prime})=0 \quad \text{for all $i>0$ 
and $z,z^{\prime}\in\C^n$}\,.\leqno{(\ast)}
\]
To see this, we extend $F^{(i)}_{x}$ to a section ${F}^{(i)}(x,y)$ in the neighbourhood of the diagonal of $X\times X$ and consider the poinwise adjoint
$\wi{F}^{(i)}(x,y)=({F}^{(i)}(y,x))^*$. By using a cut-off function $\eta$ in the neighbourhood of the diagonal we define operators
$F^{(i)}P_{p}$ and
$P_{p}\wi{F}^{(i)}$ having kernels
$$\eta(d(x,y))F^{(i)}(x,y)P_{p}(x,y)\quad \mbox{ and }\quad
\eta(d(x,y))P_{p}(x,y)\wi{F}^{(i)}(x,y)$$
with respect to $dv_X(y)$. Set
\[
\cT_{p}= T_{p}
-\sum_{i\leqslant\deg F_{x}} (F^{(i)}P_{p})\, p^{i/2}.
\]
It turns out that there exists $C>0$ such that for every $p>p_0$ and
$s\in L^2(X, L^p\otimes E)$ we have
\[
\norm{\cT_p\,s}_{L^2}\leqslant C p^{-1/2}\norm{s}_{L^2}\,,\quad
\|\cT_{p}^* s  \|_{L^2} \leqslant C p^{-1/2} \|s\|_{L^2}\,. \leqno{(\ast\ast)}
\]
Using $(\ast\ast)$  and comparing the the Taylor development of $\wi{F}^{(i)}$ 
in normal coordinates around $x$ to the expansion of the Bergman kernel we obtain in \cite[Lemma\,4.14]{MM08b}
\begin{equation*}\label{a6.41}
\frac{\partial ^\alpha \wi{F}^{(i)}}
{\partial Z^{\prime \alpha}}(x,0)=0\,, \quad \text{for} \,\,
  i-|\alpha|\geqslant j>0.
\end{equation*}
This implies (cf.\ \cite[Lemma\,4.15]{MM08b})
\begin{equation*}\label{6.381}
F^{(i)}_x(0,\ov{z}^{\prime})=0 \quad \text{for all $i>0$ 
and all $z^{\prime}\in\C^n$}\,.
\end{equation*}
The latter identity yields $(\ast)$ (cf.\ \cite[Lemma\,4.16]{MM08b}) and hence Proposition \ref{toet3.2}.  

Coming back to the proof of \eqref{toe3.60}, let us compare the asymptotic expansion of $T_p$ and
$T_{g_0,\,p}=P_p\,g_0\,P_p$. Using the Notation \ref{noe2.7},
the expansion \eqref{toe2.13} (for $k=1$) reads
\begin{equation}\label{6.33a}
p^{-n}T_{g_0,\,p,\,x_0}(Z,Z^\prime)  \cong
(g_{0}(x_0) \cP_{x_0}
+Q_{1,\,x_0}(g_0)\cP_{x_0}\,p^{-1/2})(\sqrt{p}Z,\sqrt{p}Z^\prime)
+\mO(p^{-1})\,,
\end{equation}
since $\mQ_{0,\,x_0}(g_0)=g_0(x_0) $ by \eqref{toe2.15}.
The expansion \eqref{toe3.2} (also for $k=1$) takes the form
\begin{equation}\label{6.33b}
p^{-n} T_{p,\,x_0} \cong (g_{0}(x_0) \cP_{x_0}
+\mQ_{1,\,x_0}\cP_{x_0}\,p^{-1/2})(\sqrt{p}Z,\sqrt{p}Z^\prime)+\mO(p^{-1})\,,
\end{equation}
where we have used Proposition \ref{toet3.2} and the definition \eqref{toe3.5}
of $g_0$.
Thus, subtracting \eqref{6.33a} from \eqref{6.33b} we obtain
\begin{equation}\label{6.33d}
p^{-n} (T_{p} -  T_{g_0,\,p})_{x_0}(Z,Z^\prime)
 \cong \big((\mQ_{1,\,x_0}-Q_{1,\,x_0}(g_0))\cP_{x_0}\big)
(\sqrt{p}Z,\sqrt{p}Z^\prime)\,p^{-1/2}+\mO(p^{-1})\,.
\end{equation}
Thus it suffices to prove:
\begin{equation}\label{6.33e}
F_{1,\,x}:=\mQ_{1,\,x}-Q_{1,\,x}(g_0)\equiv0\,.
\end{equation}
which is done in \cite[Lemma\,4.18]{MM08b}. This finishes the proof of \eqref{toe3.60} and \eqref{toe3.6}.
Hence the expansion
\eqref{toe3.4} of $T_p$  holds for $m=0$. Moreover, if $T_{p}$ is self-adjoint, then from (4.70), (4.71) follows that $g_{0}$ is also self-adjoint.  

We show inductively that \eqref{toe3.4} holds for every $m\in\N$. To handle \eqref{toe3.4}
for $m=1$ let us consider the operator $p(T_p-P_pg_0P_p)$.
The task is to show that $p\big(T_p-T_{g_0,\,p}\big)$
satisfies the hypotheses of Theorem \ref{toet3.1}.
The first two conditions are easily verified. To prove the third,
just subtract the asymptotics of
$T_{p,\,x_0}(Z,Z^\prime)$ (given by \eqref{toe3.2}) and
$T_{g_0,\,p,\,x_0}(Z,Z^\prime)$ (given by \eqref{toe2.13}).
Taking into account Proposition \ref{toet3.2} and \eqref{6.33e},
the coefficients of $p^0$ and
$p^{-1/2}$ in the difference vanish, which yields the desired conclusion.
Proposition \ref{toet3.2} and \eqref{toe3.6} applied to
$p(T_p-P_pg_0P_p)$ yield $g_1\in \cC^\infty(X,\End(E))$
such that \eqref{toe3.4} holds true for $m=1$.

We continue in this way the induction process to get \eqref{toe3.4}
for any $m$. This completes the proof of Theorem \ref{toet3.1}.
\end{proof}
Recall that the Poisson bracket 
$\{ \,\cdot\, , \,\cdot\, \}$
on $(X,2\pi \om)$ is defined as follows.  For $f, g\in \cC^\infty (X)$,
let $\xi_{f}$ be the Hamiltonian vector field generated by $f$, 
which is defined by $2 \pi i_{\xi_{f}}\om=df$. Then 
\begin{equation}\label{toe4.1a}
\{f, g\}:= \xi_{f}(dg).
\end{equation}
\begin{theorem}[{\cite[Th.\,1.1]{MM08b}, 
\cite[Th.\,7.4.1]{MM07}}]\label{toet4.1b}
The product of the Toeplitz operators 
$T_{f,\,p}$ and $T_{g,\,p}$, with $f,g\in\cC^\infty(X,\End(E))$, 
is a Toeplitz operator, i.e., it admits the asymptotic expansion
in the sense of \eqref{atoe2.1}{\rm:}
\begin{equation}\label{toe4.2b}
T_{f,\,p}\,T_{g,\,p}=\sum^\infty_{r=0}p^{-r}T_{C_r(f,g),\,p}+\mO(p^{-\infty}),
\end{equation} 
where $C_r$ are bi-differential operators, $C_r(f,g)\in\cC^\infty(X,\End(E))$,
$C_0(f,g)=fg$.

\noindent
If $f,g\in(\cC^\infty(X),\{\cdot,\cdot\})$ with the Poisson bracket
defined in \eqref{toe4.1a}, we have 
\begin{equation}\label{toe4.4}
\big[T_{f,\,p}\,,T_{g,\,p}\big]=\frac{\sqrt{-1}}{\, p}T_{\{f,\,g\},\,p}+\mO(p^{-2}).
\end{equation} 
\end{theorem}

\begin{proof}
Firstly, it is obvious that $P_p\,T_{f,\,p}\,T_{g,\,p}\,P_p=T_{f,\,p}\,T_{g,\,p}$. 
Lemmas \ref{toet2.1} and \ref{toet2.3} imply
$T_{f,\,p}\,T_{g,\,p}$ verifies \eqref{toe3.1}.
Like in \eqref{toe2.17}, we have
for $Z,Z^\prime \in T_{x_0}X$, $\abs{Z},\abs{Z^{\prime}}<\var/4$:
\begin{multline}\label{toe4.5}
(T_{f,\,p}\,T_{g,\,p})_{x_0}(Z,Z^\prime)= \int_{{ T_{x_0}X}}
T_{f,\,p,\,x_0}(Z,Z^{\prime\prime})
\rho(4|Z^{\prime\prime}|/\var)
T_{g,\,p,\,x_0}(Z^{\prime\prime},Z^{\prime}) \\
\times \kappa_{x_0}(Z^{\prime\prime})
\,dv_{TX}(Z^{\prime\prime})+ \cO(p^{-\infty}).
\end{multline}  
By Lemma \ref{toet2.3} and \eqref{toe4.5}, we deduce as in the proof of 
Lemma \ref{toet2.3}, that for $Z,Z^\prime \in T_{x_0}X$, 
$\abs{Z},\abs{Z^{\prime}}<\var/4$, we have
\begin{align} \label{toe4.6}
p^{-n}(T_{f,\,p}\,T_{g,\,p})_{x_0}(Z,Z^\prime)\cong 
\sum^k_{r=0}(Q_{r,\,x_0}(f,g)\cP_{x_0})(\sqrt{p}\,Z,\sqrt{p}\,Z^{\prime})
p^{-\frac{r}{2}} + \mO(p^{-\frac{k+1}{2}}),
\end{align}
and with the notation \eqref{toe1.6}, 
\begin{align} \label{toe4.7}
Q_{r,\,x_0}(f,g)= \sum_{r_1+r_2=r}  \cK[Q_{r_1,\,x_0}(f), Q_{r_2,\,x_0}(g)].
\end{align}
Thus $T_{f,\,p}\,T_{g,\,p}$ is a Toeplitz operator by Theorem \ref{toet3.1}.
Moreover, it follows from the proofs of Lemma \ref{toet2.3} 
and Theorem \ref{toet3.1}
that $g_l=C_l(f,g)$, where $C_l$ are bi-differential operators. 

{}From \eqref{toe1.6}, \eqref{toe2.15} and \eqref{toe4.7}, we get
\begin{equation} \label{toe4.8}
C_0(f,g)(x)= Q_{0,\,x}(f,g)
= \cK[Q_{0,\,x}(f), Q_{0,\,x}(g)]
= f(x)g(x)\,.
\end{equation}
The commutation relation \eqref{toe4.4} follows from 
\begin{equation}\label{toe4.16}
C_1(f,g)(x)-C_1(g,f)(x)=\imat\{f,g\} \,\Id_{E}\,.
\end{equation} 
There are two ways to prove \eqref{toe4.16}. One is to compute directly the difference and to use some of the identities \eqref{toe1.12}. This method works also for symplectic manifolds, see \cite[p.\,593-4]{MM08b}, \cite[p.\,311]{MM07}.
On the other hand, in the K\"ahler case one can compute explicitly each coefficient
$C_1(f,g)$ (which in the general symplectic case is more involved), and then take the difference. 
The formula for $C_1(f,g)$ is given in the next theorem.
This finishes the proof of Theorem \ref{toet4.1}.
\end{proof}
\begin{theorem}[{\cite[Th.\,0.3]{MM10}}]\label{toec1}
Let $f,g\in\cC^\infty(X,\End(E))$. We have
 \begin{equation}\label{toe4.3} \begin{split}
 C_0(f,g)=&fg, \\
C_1(f,g)=&-  \frac{1}{2\pi}
\big\langle \nabla^{1,0} f, \ov{\partial}^{E} g\big\rangle_{\om}\in 
\cC^{\infty}(X,\End(E)),\\
C_2(f,g)=& \, \bb_{2,\,f,\,g} - \bb_{2,\,fg}- \bb_{1,\,C_1(f,\,g)}.
\end{split}\end{equation}
If $f,g\in \cC^\infty(X)$, then
 \begin{align}\label{toe4.3a} \begin{split}
C_2(f,g)= &\,
\frac{1}{8\pi^2 } \left\langle  D^{1,0}\partial f, D^{0,1}\ov{\partial} g\right\rangle
+ \frac{\sqrt{-1}}{4\pi^2 } \left\langle  \Ric_\om, 
\partial f \wedge\ov{\partial} g\right\rangle\\
 &-\frac{1}{4\pi^2} \left\langle\partial f\wedge \ov{\partial} g, R^E\right\rangle_{\om}\,.
\end{split}\end{align}
\end{theorem}
The next result and Theorem \ref{toet4.1b} show that 
the Berezin-Toeplitz quantization has the correct semi-classical behavior.
\begin{theorem}\label{toet4.2} For $f\in \cC^\infty(X, \End(E))$,
the norm of $T_{f,\,p}$ satisfies
\begin{equation}\label{toe4.17}
\lim_{p\to\infty}\norm{T_{f,\,p}}={\norm f}_\infty
:=\sup_{0\neq u\in E_x,\, x\in X} |f(x)(u)|_{h^{E}}/ |u|_{h^{E}}.
\end{equation}
\end{theorem}
\begin{proof}
Take a point $x_0\in X$ and $u_0\in E_{x_0}$ with $|u_0|_{h^{E}}=1$ 
such that $|f(x_0)(u_0)|={\norm f}_\infty$.
Recall that in Section \ref{s3.2}, we trivialized the bundles $L$, $E$ 
in normal coordinates near $x_0$, 
and $e_L$ is the unit frame of $L$ which trivializes $L$. 
Moreover, in this normal coordinates, $u_0$ is a trivial section of $E$.
Considering the sequence of sections 
$S^p_{x_0}=p^{-n/2}P_p(e_L^{\otimes p}\otimes u_0)$, 
we have by \eqref{toe2.9},
\begin{equation}\label{toe4.18}
\big\|T_{f,\,p}\,S^p_{x_0}-f(x_0)S^p_{x_0}\big\|_{L^2}
\leqslant \frac{C}{\sqrt{p}}\norm{S^p_{x_0}}_{L^2}\,, 
\end{equation}
which immediately implies \eqref{toe4.17}.
\end{proof}
Note that if $f$ is a real function, then $df(x_0)=0$, so we can improve 
the bound  $Cp^{-1/2}$ in \eqref{toe4.18} to $Cp^{-1}$.

\begin{remark}\label{toer1}
\textbf{(i)} Relations \eqref{toe4.4} and \eqref{toe4.17} were first proved in some special cases:
in \cite{KliLe:92} for Riemann surfaces, in \cite{Cob:92} for $\C^n$
and in \cite{BLU:93} for bounded symmetric domains in $\C^n$,
by using explicit calculations.
Then Bordemann, Meinrenken and Schlichenmaier \cite{BMS94}
treated the case of a compact K{\"a}hler manifold (with $E=\C$)
using the theory of Toeplitz structures (generalized Szeg{\"o} operators) by
Boutet de Monvel and Guillemin \cite{BouGu81}.
Moreover, Schlichenmaier \cite{Schlich:00}
(cf. also  \cite{KS01}, \cite{Charles03})
continued this train of thought and showed that for any $f,g\in \cC^\infty(X)$,
 the product $T_{f,\,p}\,T_{g,\,p}$ has an asymptotic expansion
 \eqref{toe4.2} and constructed geometrically an associative star product.
\\[2pt]
\textbf{(ii)}
The construction of the star-product can be carried out even in the presence of a twisting vector bundle $E$.
Let $f,g\in\cC^\infty(X,\End(E))$. Set
\begin{equation}\label{toe4.4c}
f*_{\hbar}g:=\sum_{k=0}^\infty C_k(f,g) \hbar^{k}\in\cC^\infty(X,\End(E))[[\hbar]]\,,
\end{equation}
where $C_{r}(f,g)$ are determined by \eqref{toe4.2}. Then \eqref{toe4.4c} defines an associative star-product on $\cC^\infty(X,\End(E))$ called \emph{Berezin-Toeplitz star-product} (cf.\ \cite{KS01,Schlich:00} for the K\"ahler case with $E=\C$ 
and \cite{MM07,MM08b} for the symplectic case and arbitrary twisting bundle $E$).
The associativity of the star-product \eqref{toe4.4c} follows immediately 
from the associativity rule for the composition of Toeplitz operators, 
$(T_{f,\,p}\circ T_{g,\,p})\circ T_{k,\,p}=
T_{f,\,p}\circ (T_{g,\,p}\circ T_{k,\,p})$ for any  
$f,g,k\in\cC^\infty(X,\End(E))$, and from the asymptotic expansion 
\eqref{toe4.2} applied to both sides of the latter equality.

The coefficients $C_r(f,g)$, $r=0,1,2$ are given by \eqref{toe4.3}. Set 
\begin{equation}\label{toe4.5a}
\{\!\{f,g\}\!\}:=
\frac{1}{2\pi\imat}
\big(\langle \nabla^{1,0} g, \ov{\partial}^{E} f\rangle_{\om}-\langle \nabla^{1,0} f, \ov{\partial}^{E} g\rangle_{\om}\big)\,.
\end{equation}
If $fg=gf$ on $X$ we have 
\begin{equation}\label{toe4.4b}
\big[T_{f,\,p}\,,T_{g,\,p}\big]=\tfrac{\sqrt{-1}}{\, p}\,T_{\{\!\{f,g\}\!\},\,p}
+\mO\big(p^{-2}\big)\,,\quad p\to\infty.
\end{equation}
Due to the fact that $\{\!\{f,g\}\!\}=\{f,g\}$ if $E$ is trivial and 
comparing \eqref{toe4.4} to \eqref{toe4.4b},  one can regard 
$\{\!\{f,g\}\!\}$ defined in \eqref{toe4.5a} 
as a non-commutative Poisson bracket.
\subsection{Quantization of compact Hermitian manifolds}
Throughout Sections \ref{s3.2}-\ref{s3.6} we supposed that the Riemannian metric $g^{TX}$ was the metric associated to $\omega$, that is, $g^{TX}(u,v)=\omega(u,Jv)$ (or, equivalently, $\Theta=\omega$). The results presented so far still hold for a general non-K\"ahler Riemannian metric $g^{TX}$. 

Let us denote the metric associated to $\omega$ by $g^{TX}_{\om}:=\om(\cdot,J\cdot)$.
The volume form of $g^{TX}_{\om}$ is given by $dv_{X,\,\om} =(2\pi)^{-n}{\det}(\dot{R}^L) dv_{X}$ (where $dv_{X}$ is the volume form of $g^{TX}$).
Moreover, $h^E_\om:={\det} (\frac{\dot{R}^L}{2\pi})^{-1} h^E$ 
defines a metric on $E$. 
We add a subscript $\om$ to indicate the objects associated to $g^{TX}_{\om}$, $h^L$ and $h^E_{\om}$.
Hence $\left\langle\,\cdot,\cdot \right \rangle_\om$ denotes the $L^2$ Hermitian product on 
$\cC^\infty (X, L^p\otimes E)$
induced by $g^{TX}_\om$, $h^L$, $h^E_\om$. This product is equivalent to the product  
$\left\langle\,\cdot,\cdot \right \rangle$ induced by $g^{TX}$, $h^L$, $h^E$.

Moreover, $H^{0}(X, L^p\otimes E)$ does not depend on the Riemannian metric on $X$ or on the Hermitian metrics on $L$, $E$. Therefore, the orthogonal projections from 
$(\cC^\infty (X, L^p\otimes E), \left\langle\,\cdot,\cdot \right \rangle_\om)$ and 
$(\cC^\infty (X, L^p\otimes E), \left\langle\,\cdot,\cdot \right \rangle)$
onto $H^{0}(X, L^p\otimes E)$ are the same. Hence $P_p=P_{p,\,\om}$ and therefore $T_{f,\,p}=T_{f,\,p,\,\om}$ as operators. 
However, their kernels are different.
If $P_{p,\,\om}(x,x')$, $T_{f,\,\om,\,p}(x,x')$, ($x,x'\in X$), denote the smooth kernels 
of $P_{p,\,\om}$, $T_{f,\,p,\,\om}$ with respect to $dv_{X,\om}(x')$,
we have
\begin{equation}\label{bk2.95}
\begin{split}
&P_{p}(x,x')=(2\pi)^{-n} {\det}(\dot{R}^L)(x') P_{p,\,\om}(x,x')\,,\\
&T_{f,\,p}(x,x')=(2\pi)^{-n} {\det}(\dot{R}^L)(x') T_{f,\,p,\,\om}(x,x')\,.
\end{split}
\end{equation}
Now, for the kernel $ P_{p,\,\om}(x,x')$, we can apply Theorem \ref{tue17} 
since $g^{TX}_\om(\cdot, \cdot)= \om(\cdot, J\cdot)$ 
is a K\"ahler metric on $TX$.
We obtain in this way the expansion of the Bergman kernel for a non-K\"ahler Riemannian metric $g^{TX}$ on $X$, see \cite[Th.\,4.1.1,\,4.1.3]{MM07}. 
Of course, the coefficients $\bb_r$ reflect in this case the presence of $g^{TX}$. For example
\begin{align}\label{abk2.51}
 \bb_0={\det}(\dot{R}^L/(2\pi)) \Id_{E},
\end{align}
and 
\begin{equation} \label{abk2.7}
\bb_1 = \frac{1}{8\pi}\det\Big(\frac{\dot{R}^L}{2\pi}\Big)
\Big[r^X_\om -2 \Delta_\om \Big(\log({\det} (\dot{R}^L))\Big)
+ 4  \sqrt{-1} \langle R^E,\om\rangle_\om \Big]\,.
\end{equation}
Using the expansion of the Bergman kernel $P_{p,\,\om}(\cdot,\cdot)$ we can deduce the expansion of the Toeplitz operators $T_{f,\,p,\,\om}$ and their kernels, analogous to Theorem \ref{toet2.3}, Corollary \ref{toec2.1} and Theorem \ref{toet4.1b}. By \eqref{bk2.95}, the coefficients of these expansion satisfy 
\begin{equation}\label{bk2.951}
\begin{split}
&\bb_{f,\,r}=(2\pi)^{-n} {\det}(\dot{R}^L)\bb_{f,\,r,\,\om}\,,\\
&C_r(f,g)=C_{r,\,\om}(f,g)\,.
\end{split}
\end{equation}

Since $X$ is compact, \eqref{bk2.95}  
allowed to reduce the general situation considered here
to the case $\omega=\Theta$ and apply Theorem  \ref{toet4.5}. 
However, if $X$ is not compact, the trick of using \eqref{bk2.95}
does not work anymore, because the operator associated to  
$g^{TX}_\om$, $h^L$, $h^E_\om$ might not have a spectral gap.
But under the hypotheses of Theorem \ref{noncompact0} the spectral gap for $D_p$ exists, so we can extend these results to certain complete Hermitian manifolds in the next section.
\end{remark}
\subsection{Quantization of complete Hermitian manifolds}
We return to the general situation of a complete manifold already considered in \S \ref{spec-gap}.
The following result, obtained in \cite[Th.\,3.11]{MM08a}, extends 
the asymptotic expansion of the Bergman kernel to complete manifolds.

\begin{theorem} \label{noncompact1}
Let $(X,\Theta)$ be a complete Hermitian manifold, $(L,h^L)$, $(E,h^E)$ be Hermitian holomorphic vector bundles of rank one and $r$, respectively.
Assume that the hypotheses of Theorem \ref{noncompact0} are fulfilled.
Then the kernel $P_p(x,x')$
has a full off--diagonal asymptotic expansion analogous to 
that of Theorem \ref{tue17} uniformly for any $x,x'\in K$, 
a compact set of $X$.
If $L=K_X:=\det(T^{*(1,0)}X)$ is the canonical line bundle on $X$, the first two conditions in (5.5) are to be replaced by    \begin{equation*}    \text{$h^L$ is induced by $\Theta$ and     $\sqrt{-1}R^{\det}<-\varepsilon\Theta$,   $\sqrt{-1}R^E> -C\Theta \Id_E$\,. }    
\end{equation*}
\end{theorem} 
The idea of the proof is that the spectral gap property \eqref{ell4,1} of Theorem \ref{noncompact0}
allows to generalize the analysis leading to the expansion in the compact case (Theorems \ref{tue16} and \ref{tue17}) to the situation at hand.

\begin{remark}
Consider for the moment that in Theorem \ref{noncompact1} we have $\Theta=\frac{\imat}{2\pi}R^L$.
Since in the proof of Theorem \ref{noncompact1} we use the same localization technique as in the compact
case, the coefficients $J_{r,x}$ in the expansion of the Bergman kernel (cf.\ \eqref{toe2.9}), in 
particular the coefficients $\bb_r(x)=J_{2r,x}(x)$ of the diagonal expansion have the same universal 
formulas as in the compact case. Thus the explicit formulas 
from Theorem 
\ref{toet4.5} for $\bb_1$ and $\bb_2$ remain valid in the case of the situation considered in Theorem 
\ref{noncompact1}. Moreover, in the general case when $\frac{\imat}{2\pi}R^L>\var\Theta$ (for some constant $\var>0$), the first formulas in 
\eqref{bk2.95} and \eqref{abk2.51}, \eqref{abk2.7} are still valid.
\end{remark}
Let $\cC^\infty_{\rm const}(X, \End(E))$ denote the algebra of smooth sections 
of $X$ which are constant map outside a compact set.
For any $f\in\cC^\infty_{\rm const}(X, \End(E))$, we consider the Toeplitz 
operator $(T_{f,\,p})_{p\in\N}$ as in \eqref{toe2.4}.
The following result generalizes Theorems \ref{toet4.1} and \ref{toet4.2}
to complete manifolds.  
\begin{theorem}[{\cite[Th.\,5.3]{MM08b}}]\label{toet5.1}
Let $(X,\Theta)$ be a complete Hermitian manifold, let $(L,h^L)$ and $(E,h^E)$ be Hermitian holomorphic vector bundles on $X$ of rank one and $r$, respectively.
Assume that the hypotheses of Theorem \ref{noncompact0} are fulfilled.
Let $f,g\in\cC^\infty_{\rm const}(X, \End(E))$. Then the following assertions hold:
\\[2pt]
(i) The product of the two corresponding Toeplitz operators admits 
the asymptotic expansion \eqref{toe4.2} in the sense of \eqref{atoe2.1}, 
where $C_r$ are bi-differential operators, 
especially, $\supp (C_r(f,g))\subset \supp(f)$ $\cap \supp(g)$,
and $C_0(f,g)=fg$.
\\[2pt]
(ii) If $f,g\in\cC^\infty_{\rm const}(X)$, then \eqref{toe4.4} holds.
\\[2pt]
(iii) Relation \eqref{toe4.17} also holds for any 
$f\in\cC^\infty_{\rm const}(X, \End(E))$.
\\[2pt]
(iv) The coefficients $C_{r}(f,g)$ are given by $C_{r}(f,g)=C_{r,\,\om}(f,g)$, where $\om=\frac{\imat}{2\pi}R^{L}$ {\rm(}compare \eqref{bk2.951}{\rm)}. 
\end{theorem}

\section{Berezin-Toeplitz quantization on K{\"a}hler orbifolds} \label{pbs4}
In this Section we review the theory of Berezin-Toeplitz quantization on K\"ahler orbifolds, especially 
we show that set of Toeplitz operators forms an algebra.
Note that the problem of quantization of orbifolds appears naturally in the study of the phenomenon of 
``quantization commutes to reduction'', since the reduced spaces are often orbifolds, see e.g.\ \cite{Ma10},
or in the problem of quantization of moduli spaces.
 
Complete explanations and references for Sections \ref{pbs4.1} and \ref{pbs4.2} are contained in \cite[\S 5.4]{MM07}, \cite[\S6]{MM08b}. Moreover, we treat there also the case of symplectic orbifolds.

This Section is organized as follows. 
In Section \ref{pbs4.1} we recall the basic definitions about orbifolds.
In Section \ref{pbs4.2} we explain the asymptotic expansion
of Bergman kernel on complex orbifolds \cite[\S 5.2]{DLM04a}, which we apply in
Section \ref{pbs4.3} to derive the Berezin-Toeplitz quantization on K{\"a}hler orbifolds. 

\subsection{Preliminaries about orbifolds}\label{pbs4.1}
We begin by the definition of orbifolds.
We define at first a  category $\mathcal{M}_s$ as follows~: 
The objects of $\mathcal{M}_s$
are the class of pairs $(G,M)$ where $M$ is a connected smooth manifold
and $G$ is a finite group acting effectively on $M$
(i.e., if $g\in G$ such that $gx=x$ for any $x\in M$, 
then $g$ is the unit element of $G$). 
Consider two objects $(G,M)$ and $(G',M')$. For $g\in G', \varphi \in \Phi $, we define
$g\varphi : M \rightarrow M'$
by $(g\varphi)(x) = g(\varphi(x))$ for $x\in M$.
A morphism $\Phi: (G,M)\rightarrow
(G',M')$ is a family of open embeddings $\varphi: M\rightarrow M'$
satisfying~:
\\[2pt]
{i)} For  each $\varphi \in \Phi $, there is an injective group
 homomorphism
$\lambda_{\varphi}~: G\rightarrow G' $ that makes $\varphi$ be
$\lambda_{\varphi}$-equivariant.
\\[2pt]
{ii)} 
If $(g\varphi)(M) \cap \varphi(M) \neq \emptyset$, then $g\in  \lambda_{\varphi}(G)$.
\\[2pt]
{ iii)} For $\varphi \in \Phi $, we have $\Phi = \{g\varphi, g\in G'\}$.

\begin{definition}[Orbifold chart, atlas, structure]\label{pbt4.1} Let $X$ be a paracompact Hausdorff space.
An $m$-dimensional \emph{orbifold chart} on $X$ consists of a
connected open set $U$ of $X$,
an object $(G_U,\widetilde{U})$ of $\mathcal{M}_s$ with $\dim\widetilde{U}=m$,
and a ramified covering $\tau_U:\widetilde{U}\to U$ which is 
$G_U$-invariant and induces a homeomorphism
$U \simeq \widetilde{U}/G_U$. We denote the chart by 
$(G_U,\widetilde{U})\stackrel{\tau_U}{\longrightarrow}U$. 

An $m$-dimensional \emph{orbifold atlas} $\mathcal{V}$ on $X$
consists of a family of $m$-dimensional orbifold charts 
$\mathcal{V} (U)
=((G_U,\widetilde{U})\stackrel{\tau_U}{\longrightarrow} U)$ 
satisfying the following conditions\,:
\\[2pt]
{(i)} The open sets $U\subset X$ form a covering $\mathcal{U}$ with
 the property:
\begin{equation}\label{pb4.1}
\begin{split}
&\text{For any $U, U'\in \mathcal{U}$ and $x\in U\cap U'$, there exists  
$U''\in \mathcal{U}$}\\
&\text{ such that $x\in U''\subset U\cap U'$}.
\end{split}
\end{equation}
{(ii)} For any $U, V\in \mathcal{U}, U\subset V$ there exists a morphism
$\varphi_{VU}:(G_U,\widetilde{U})\rightarrow (G_V,\widetilde{V})$, 
which covers the inclusion $U\subset V$ and satisfies 
$\varphi_{WU}=\varphi_{WV} \circ \varphi_{VU}$
for any $U,V,W\in \mathcal{U}$, with $U\subset V \subset W$. 

It is easy to see that there exists a unique maximal orbifold atlas 
$\mathcal{V}_{\rm max}$ containing $\mathcal{V}$;
$\mathcal{V}_{\rm max}$ consists of all orbifold charts 
$(G_U,\widetilde{U})\stackrel{\tau_U}{\longrightarrow} U$,
which are locally isomorphic to charts from $\mathcal{V}$ in the neighborhood 
of each point of $U$. A maximal orbifold atlas $\mathcal{V}_{\rm max}$ 
is called an \emph{orbifold structure} and the pair $(X,\mathcal{V}_{\rm max})$ 
is called an orbifold. As usual, once we have an orbifold atlas 
$\mathcal{V}$ on $X$ we denote the orbifold by $(X,\mathcal{V})$, 
since $\mathcal{V}$ uniquely determines $\mathcal{V}_{\rm max}$\,.
\end{definition}
In Definition \ref{pbt4.1} we can replace $\mathcal{M}_s$ by a category 
of manifolds with an additional structure such as orientation, 
Riemannian metric, almost-complex structure or complex structure.
We impose that the morphisms (and the groups) preserve the specified
structure. So we can define oriented, Riemannian,
almost-complex or complex orbifolds.
\begin{definition}[regular and singular set]\label{pbt4.2}
Let $(X,\mathcal{V})$ be an orbifold. For each $x\in X$, we can choose a
small neighborhood $(G_x, \widetilde{U}_x)\to U_x$ such
that $x\in \widetilde{U}_x$ is a fixed point of $G_x$
(it follows from the definition that such a $G_x$ is unique up to 
isomorphisms for each $x\in X$). 
We denote by  $|G_x|$ the cardinal of $G_x$.
If $|G_x|=1$, then $X$ has a smooth manifold structure in the neighborhood 
of $x$, which is called a smooth point of $X$. 
If  $|G_x|>1$, then $X$ is not a smooth manifold in the neighborhood of $x$, 
which is called a singular point of $X$. We denote by
$X_{\rm sing}= \{x\in X; |G_x|>1\}$ the singular set of $X$,
and $X_{\rm reg}= \{x\in X; |G_x|=1\}$ the regular set of $X$.
\end{definition}
It is useful to note that on an orbifold $(X,\mathcal{V})$ we can 
construct partitions of unity. First, let us call a function on $X$ smooth, 
if its lift to any chart of the orbifold atlas $\mathcal{V}$ is smooth 
in the usual sense. Then the definition and construction of a smooth 
partition of unity associated to a locally finite covering carries over 
easily from the manifold case. The point is to construct 
smooth $G_U$-invariant functions with compact support on
$(G_U,\widetilde{U})$.
\begin{definition}[Orbifold Riemannian metric]\label{pbt4.3}
Let $(X,\mathcal{V})$ be an arbitrary orbifold. 
A \emph{Riemannian metric} on $X$ is a Riemannian metric $g^{TX}$
on $X_{\rm reg}$ such that the lift of $g^{TX}$ to any chart of the orbifold 
atlas $\mathcal{V}$ can be extended to a smooth Riemannian metric. 
\end{definition}
Certainly, for any $(G_U, \wi{U})\in \mathcal{V}$, we can always 
construct a $G_U$-invariant Riemannian metric on $\wi{U}$. 
By a partition of unity argument, we see that there exist Riemannian metrics 
on the orbifold $(X,\mathcal{V})$.
\begin{definition}\label{pbt4.4}
 An \emph{orbifold vector bundle} $E$ over an
orbifold $(X,\mathcal{V})$ is defined as follows\,: $E$ is an orbifold
and for any $U\in \mathcal{U}$, $(G_U^{E}, \widetilde{p}_U:
\widetilde{E}_U \rightarrow \widetilde{U})$ is
 a $G_U^{E}$-equivariant vector bundle and $(G_U^{E}, \widetilde{E}_U)$
(resp. $(G_U=G_U^{E}/K_U^{E}, \widetilde{U})$, where
$K_U^{E}= \ker (G_U^{E}\rightarrow\mbox{{\rm Diffeo}} (\widetilde{U})))$
is the orbifold structure of $E$ (resp.\ $X$). 
If $G_U^E$ acts effectively on $\widetilde{U}$ for $U\in \mathcal{U}$, 
i.e.\ $K_U^{E} = \{ 1\}$, we call $E$ a proper orbifold vector bundle.
\end{definition}

Note that any structure on $X$ or $E$ is locally
$G_x$ or $G_{U_x}^{E}$-equivariant.

Let $E$ be an orbifold vector bundle on $(X,\mathcal{V})$. 
For $U\in \mathcal{U}$,
let $\wi{E ^{\pr}_U}$ be the maximal $K_U^{E}$-invariant sub-bundle of
 $\wi{E}_U$ on $\wi{U}$. Then $(G_U, \wi{E ^{\pr}_U})$  defines a 
proper orbifold vector bundle on $(X, \mathcal{V})$, denoted by $E ^{\pr}$.

The  (proper)  orbifold tangent bundle $TX$ on an orbifold $X$ is defined by 
$(G_U, T\widetilde{U} \rightarrow \widetilde{U})$, for $U\in \mathcal{U}$.
In the same vein we introduce the cotangent bundle $T^*X$.
We can form tensor products of bundles by taking the tensor
products of their local expressions in the charts of an orbifold atlas.

Let $E \rightarrow X$ be an orbifold vector bundle and 
$k\in \N\cup\{\infty\}$. A section $s: X\rightarrow E$ is called ${\cC}^k$
 if for each $U\in \mathcal{U}$, $s|_{U}$ is covered by 
a $G_U^{E}$-invariant ${\cC}^k$
section $\widetilde{s}_U : \widetilde{U} \rightarrow \widetilde{E}_U$.
 We denote by $\cC^k(X,E)$  
the space of $\cC^k$ sections of $E$ on $X$. 
\\[2pt]
\textbf{\emph{Integration on orbifolds.\/}}
If $X$ is oriented, we define the integral
 $\int_{X} \alpha$ for a form $\alpha$
 over $X$ (i.e. a section of $ \Lambda  (T^*X)$ over $X$) as follows. 
If $\supp(\alpha) \subset U\in \mathcal U$ set
\begin{equation}\label{pb4.5}
\int_{X} \alpha: = \frac{1}{|G_U|} \int_{\widetilde{U}}
\widetilde{\alpha}_U.
\end{equation}
It is easy to see that the definition is independent of the chart. For general $\alpha$ we extend the definition by using a partition of unity.

If $X$ is an oriented Riemannian orbifold,
there exists a canonical volume element $dv_X$ on $X$, which is a section of 
$\Lambda^m(T^*X)$, $m=\dim X$. Hence, we can also integrate functions on $X$.
\\[2pt]
\textbf{\emph{Metric structure on orbifolds.\/}}
Assume now that the Riemannian orbifold $(X,\mathcal V)$ is compact. We define a metric on $X$ by setting for $x,y\in X$,
\begin{eqnarray}\begin{array}{l}
d(x,y) = \mbox{Inf}_\gamma  \Big \{ \sum_i \int_{t_{i-1}}^{t_i}
|\frac{\partial }{\partial t}\widetilde{\gamma}_i(t)| dt \Big |
 \gamma: [0,1] \to X, \gamma(0) =x, \gamma(1) = y,\\
 \hspace*{15mm}  \mbox{such that there exist }   t_0=0< t_1 < \cdots < t_k=1,
\gamma([t_{i-1}, t_i])\subset U_i,  \\
\hspace*{15mm}U_i \in \mathcal{U}, \mbox{ and a } \cC^{\infty}
\mbox{ map  } \widetilde{\gamma}_i: [t_{i-1}, t_i] \to \widetilde{U}_i
 \mbox{ that covers } \gamma|_{[t_{i-1}, t_i]}   \Big \}.
\end{array}\nonumber
\end{eqnarray}
Then $(X, d)$ is a metric space.
For $x\in X$, set $d(x,X_{\rm sing}):=\inf_{y\in X_{\rm sing}} d(x,y)$.
\\[2pt]
\textbf{\emph{Kernels on orbifolds.\/}}
Let us discuss briefly kernels and operators on orbifolds.
For any open set $U\subset X$ and orbifold chart
$(G_U,\widetilde{U})\stackrel{\tau_U}{\longrightarrow} U$, 
we will add a superscript $\, \wi{}\,$ to indicate the corresponding
objects on $\widetilde{U}$.
Assume that  
$\wi{\mK}(\wi{x},\wi{x}^{\,\prime})\in \cC^\infty(\wi{U} \times \wi{U},
\pi_1^* \wi{E}\otimes \pi_2^* \wi{E}^*)$
verifies 
\begin{align}\label{pb4.3}
(g,1)\wi{\mK}(g^{-1}\wi{x},\wi{x}^{\,\prime})
=(1,g^{-1})\wi{\mK}(\wi{x},g\wi{x}^{\,\prime})
\quad \text{ for any $g\in G_{U}$,}
\end{align}
where $(g_1,g_2)$ acts on $\wi{E}_{\wi{x}}\times \wi{E}_{\wi{x}^{\,\prime}}^*$ 
by $(g_1,g_2)(\xi_1,\xi_2)= (g_1\xi_1,g_2 \xi_2)$. 

We define the operator $\wi{\mK}: \cC^\infty_0(\wi{U}, \wi{E})
\to \cC^\infty(\wi{U}, \wi{E})$ by
\begin{equation}\label{pb4.4}
(\wi{\mK}\, \wi{s}) (\wi{x})= \int_{\wi{U}} \wi{\mK}(\wi{x},\wi{x}^{\,\prime}) 
 \wi{s}(\wi{x}^{\,\prime}) dv_{\wi{U}} (\wi{x}^{\,\prime}) 
\quad\text{for $\wi{s} \in \cC^\infty_0(\wi{U}, \wi{E})$\,.}
\end{equation}
For $\wi{s} \in \cC^\infty(\wi{U}, \wi{E})$ and $g\in G_U$,
$g$ acts on $\cC^\infty(\wi{U}, \wi{E})$ by: 
$(g \cdot \wi{s})(\wi{x}):=g \cdot \wi{s}(g^{-1}\wi{x})$. We can then identify 
an element $s \in \cC^\infty({U}, {E})$ 
with an element $\wi{s} \in \cC^\infty(\wi{U}, \wi{E})$ 
verifying $g\cdot \wi{s}=\wi{s}$ for any $g\in G_U$.

With this identification, we define the operator 
$\mK: \cC^\infty_0(U,{E})\to \cC^\infty({U},{E})$ by 
\begin{equation}\label{pb4.6}
({\mK} s)(x)=\frac{1}{|G_U|}\int_{\wi{U}} \wi{\mK}(\wi{x},\wi{x}^{\,\prime}) 
 \wi{s}(\wi{x}^{\,\prime}) dv_{\wi{U}} (\wi{x}^{\,\prime}) 
\quad\text{for $s \in \cC^\infty_0({U},{E})$\,,}
\end{equation}
where $\wi{x}\in\tau^{-1}_U(x)$.
Then the smooth kernel $\mK(x,x^\prime)$ of the operator $\mK$ 
with respect to $dv_X$ is
\begin{align}\label{pb4.7}
\mK(x,x^\prime)= \sum_{g\in G_U} (g,1)\wi{\mK}(g^{-1}\wi{x},\wi{x}^{\,\prime}).
\end{align}

Let $\mK_1,\mK_2$ be two operators as above and assume that the kernel of 
one of $\wi{\mK}_1, \wi{\mK}_2$ has compact support.
By \eqref{pb4.5}, \eqref{pb4.3} and \eqref{pb4.6}, 
the kernel of $\mK_1\circ \mK_2$ is given by
\begin{equation}\label{pb4.9}
(\mK_1\circ \mK_2)(x,x^\prime)= \sum_{g\in G_U} 
(g,1)(\wi{\mK}_1\circ \wi{\mK}_2)(g^{-1}\wi{x},\wi{x}^{\,\prime}).
\end{equation}

\subsection{Bergman kernel on K{\"a}hler orbifolds} \label{pbs4.2}
In this section we study the asymptotics of the Bergman kernel on orbifolds.
\\[2pt]
\textbf{\emph{Dolbeault cohomology of orbifolds.\/}}
Let $X$ be a compact complex orbifold of complex dimension $n$
with complex structure $J$.
Let $E$ be a holomorphic orbifold vector bundle on $X$.

Let $\cO_X$ be the sheaf over $X$ of local $G_U$-invariant holomorphic
functions over $\widetilde{U}$, for $U\in {\mathcal U}$.
The local $G^{E}_U$ -invariant holomorphic sections of
$\widetilde{E} \rightarrow \widetilde{U}$ define 
a sheaf $\cO_X(E)$ over $X$.
Let $H^\bullet(X, \cO_X(E))$ be the cohomology of the sheaf
$\cO_X(E)$ over $X$. Notice that by Definition, we have
$\cO_X(E)=\cO_X(E ^{\pr})$.
Thus without lost generality, we may and will assume that 
$E$ is a proper orbifold vector bundle on $X$.

Consider a section $s\in\cC^\infty(X,E)$ and a local section 
$\wi{s}\in\cC^\infty(\wi{U},\wi{E}_U)$ covering $s$.
Then $\db^{\wi{E}_U}\wi{s}$ covers a section of $T^{*(0,1)}X\otimes E$ 
over $U$, denoted $\db^Es|_U$.
The family of sections $\{\db^Es|_U\,:\,U\in\mathcal{U}\}$ patch together 
to define a global section $\db^Es$ of $T^{*(0,1)}X\otimes E$ over $X$. 
In a similar manner we define $\overline\partial^E\alpha$ for a 
$\cC^{\infty}$ section $\alpha$ of $\Lambda (T^{*(0,1)}X) \otimes E$ over $X$. 
We obtain thus the Dolbeault complex 
($\Omega^{0,\bullet}(X,E), \overline{\partial}^E$)\,:
\begin{equation}\label{pb4.12}
0 \longrightarrow \Omega^{0,0}(X,E)
\stackrel{\overline{\partial}^E}{\longrightarrow} \cdots
\stackrel{\overline{\partial}^E}{\longrightarrow} \Omega^{0,n}(X,E)
\longrightarrow 0 .
\end{equation}
{}From the abstract de Rham theorem there exists a canonical isomorphism
\begin{eqnarray}\label{pb4.13}
H^\bullet(\Omega^{0,\bullet}(X,E), \overline{\partial}^E) \simeq
H^\bullet(X,\cO_X(E)).
\end{eqnarray}
In the sequel, we also denote  $H^\bullet(X,\cO_X(E))$ 
by $H^\bullet(X,E)$. 
\\[2pt]
\textbf{\emph{Prequantum line bundles.\/}}
We consider a complex orbifold $(X,J)$ endowed with the complex structure $J$. 
Let $g^{TX}$ be a Riemannian metric on $TX$ compatible with $J$.
There is then an associated $(1,1)$-form $\Theta$ given by 
$\Theta(U,V)=g^{TX}(JU,V)$. The metric $g^{TX}$ is called a K{\"a}hler metric 
and the orbifold $(X,J)$ is called a \emph{K{\"a}hler orbifold}
if $\Theta$ is a closed form, that is, $d\Theta=0$.
In this case $\Theta$ is a symplectic form, called K{\"a}hler form. We will 
denote the K{\"a}hler orbifold by $(X,J,\Theta)$ or shortly by $(X,\Theta)$.

Let $(L,h^L)$ be a holomorphic Hermitian proper orbifold line bundle 
on an orbifold $X$, and let $(E,h^E)$ be a holomorphic Hermitian proper 
orbifold vector bundle on $X$.

We assume that the associated curvature $R^L$ of $(L,h^L)$ verifies 
\eqref{prequantum}, i.e., $(L,h^L)$ is a positive proper orbifold line bundle 
on $X$. This implies that $\om:=\frac{\sqrt{-1}}{2\pi}R^L$ is 
a K{\"a}hler form on $X$, $(X,\om)$ is a K{\"a}hler orbifold and 
$(L,h^L,\nabla^L)$ is a prequantum line bundle on $(X,\om)$.

Note that the existence of a positive line bundle $L$ on 
a compact complex orbifold $X$ implies that the Kodaira map
associated to high powers of $L$ gives a holomorphic embedding of $X$ 
in the projective space. This is the generalization due to Baily of 
the Kodaira embedding theorem (see e.g.\ \cite[Theorem\,5.4.20]{MM07}).
\\[2pt]
\textbf{\emph{Hodge theory.\/}}
Let $g^{TX}=\om(\cdot,J\cdot)$ be the Riemannian metric on $X$
induced by $\om=\frac{\sqrt{-1}}{2\pi}R^L$.
Using the Hermitian product along the fibers of $L^p$, $E$, 
$\Lambda(T^{*(0,1)}X)$, the Riemannian volume form $dv_X$ and 
the definition \eqref{pb4.5} of the integral on an orbifold, we introduce 
an $L^2$-Hermitian product on $\Omega^{0,\bullet}(X,L^p\otimes E)$ 
similar to \eqref{lm2.0}. This allows to define the formal adjoint 
$\overline{\partial}^{L^p\otimes E,*}$ of $\overline{\partial}^{L^p\otimes E}$ 
and the operators $D_p$ and $\square_p$ as in \eqref{lm2.1}\,.
Then $D_p^2$ preserves the $\Z$-grading of 
$\Omega^{0,\bullet} (X,L^p\otimes E)$.
We note that Hodge theory extends to compact orbifolds and delivers a canonical isomorphism
\begin{equation}\label{pb4.16}
H^{q}(X,L^p\otimes E)\simeq \Ker(D_p^2|_{\Omega^{0,q}}).
\end{equation}
\\[2pt]
\textbf{\emph{Spectral gap.\/}}
By the same proof as in \cite[Theorems\,1.1,\,2.5]{MM02}, 
\cite[Theorem 1]{BVa89},
we get vanishing results and the spectral gap property.
\begin{theorem}\label{pbt4.11} 
Let $(X,\om)$ be a compact K{\"a}hler orbifold, 
$(L,h^L)$ be a prequantum holomorphic Hermitian proper orbifold line bundle 
on $(X,\om)$ and $(E,h^E)$ be an arbitrary holomorphic Hermitian 
proper orbifold vector bundle on $X$.

\noindent
Then there exists $C>0$ such that
for any $p\in \N$ 
\begin{equation}\label{pb4.17}
\spec(D_{p}^2)\subset \{0\}\, \cup\, ]4\pi p-C,+\infty[,
\end{equation}
 and $D_{p}^2|_{\Omega^{0, >0}}$ is invertible for $p$ large enough.
Consequently, we have the Kodaira-Serre vanishing theorem, namely, 
for $p$ large enough,
\begin{equation}\label{pb4.18}
H^{q}(X,L^p\otimes E)= 0\,, \quad \text{\rm for every $q>0$.}
\end{equation}
\end{theorem}
\noindent
\textbf{\emph{Bergman kernel.\/}}
As in \S \ref{s2.1}, we define the Bergman kernel as
the smooth kernel with respect to the Riemannian volume form $dv_X(x')$ of 
the orthogonal projection (Bergman projection) $P_p$ from 
$\cC^\infty(X, L^p\otimes E)$ onto $H^0(X,L^p\otimes E)$.

Let $d_p = \dim H^0(X,L^p\otimes E)$ and consider an arbitrary orthonormal 
basis $\{S^p_i\}_{i=1}^{d_p}$ of $H^0(X,L^p\otimes E)$ with respect to 
the Hermitian product \eqref{lm2.0} and \eqref{pb4.5}. 
In fact, in the local coordinate above, $\widetilde{S}^p_i(\widetilde{z})$ 
are $G_x$-invariant on $\widetilde{U}_x$, and
\begin{align}\label{pb4.19}
P_p (y,y') &= \sum_{i=1}^{d_p}  \widetilde{S}^p_i ( \widetilde{y})
 \otimes (\widetilde{S}^p_i(\widetilde{y}^\prime))^*,
\end{align}
where we use $\widetilde{y}$ to denote the point in $\widetilde{U}_x$
representing $y\in {U}_x$. 
\\[2pt]
\textbf{\emph{Asymptotics of the Bergman kernel.\/}}
The Bergman kernel on orbifolds has an asymptotic expansion, which we now describe.
We follow the same pattern as in the smooth case.
The spectral gap property \eqref{pb4.17} shows that we have the
analogue of Theorem \ref{tue16}, with the same $F$ as given in \eqref{0c3}:
\begin{equation}\label{pb4.21}
|P_p(x,x') -  F(D_p)(x,x')|_{\cC^m(X\times X)}
\leqslant C_{l,m,\var} p^{-l}.
\end{equation}

As pointed out in \cite{Ma05}, the property of the finite propagation speed
of solutions of hyperbolic equations still holds on an orbifold 
(see the proof in \cite[Appendix D.2]{MM07}).
Thus $F(D_p)(x,x')= 0$ for every for $x,x'\in X$ satisfying 
$d(x, x')\geqslant \var$. 
Likewise, given $x\in X$, $F(D_p)(x,\cdot)$ only depends on the restriction 
of $D_p$ to $B^X(x,\var)$. Thus the problem of the asymptotic expansion 
of $P_p(x,\cdot)$ is local.

For any compact set $K\subset X_{\rm reg}$, the Bergman kernel $P_p(x,x')$
has an asymptotic expansion as in Theorem \ref{tue17} by the same argument as in Theorem \ref{tue16}.

Let now $x\in X_{\rm sing}$ and let $(G_U,\widetilde{U})\stackrel{\tau_U}{\longrightarrow} U$
be an orbifold chart near $x$.
We recall that for every open set $U\subset X$ and orbifold chart
$(G_U,\widetilde{U})\stackrel{\tau_U}{\longrightarrow} U$,
we add a superscript $\, \wi{}\,$ to indicate the corresponding
objects on $\widetilde{U}$. 
Let $\partial U=\ov{U}\setminus U$, $U_1=\{ x\in U, d(x,\partial U)\geqslant\var\}$. 
Then $F(\wi{D}_p)(\wi{x},\wi{x}^{\,\prime})$  is well defined for 
$\wi{x},\wi{x}^{\,\prime}\in \wi{U}_1=\tau_U^{-1}(U_1)$. Since 
$g\cdot F(\wi{D}_p)=F(\wi{D}_p) g$, we get
\begin{equation}\label{pb4.22a}
(g, 1) F(\wi{D}_p)( g^{-1} \wi{x},\wi{x}^\prime)
= (1, g^{-1}) F(\wi{D}_p)(\wi{x}, g\wi{x}^\prime)\,,
\end{equation}
for every $g\in G_U$, 
$\wi{x},\wi{x}^{\,\prime}\in \wi{U}_1$. 
Formula \eqref{pb4.7} shows that for every $x,x^\prime\in U_1$ 
and $\wi{x},\wi{x}^{\,\prime}\in \wi{U}_1$ representing $x,x^\prime$, we have
\begin{equation}\label{pb4.22}
F(D_p)(x,x^\prime) = \sum_{g\in G_{U}}
(g, 1) F(\wi{D}_p)( g^{-1}\wi{x},\wi{x}^{\,\prime}).
\end{equation}

%
%
In view of \eqref{pb4.22}, the strategy is to use the expansion for $F(\wi{D}_p)(\cdot,\cdot)$ in order to deduce
the expansion for $F(D_p)(\cdot,\cdot)$ and then for $P_p(\cdot,\cdot)$, due to \eqref{pb4.21}.
In the present situation the kernel $\cP$ 
takes the form
\begin{equation}\label{pb4.25}
\cP(\wi{Z},\wi{Z}^\prime) =\exp\Big(-\frac{\pi}{2}\sum_i\big(|\wi{z}_i|^2
+|\wi{z}^{\prime}_i|^2 -2\wi{z}_i\overline{\wi{z}}_i^\prime\big)\Big)\,.
\end{equation}

For details we refer to \cite[\S\,5.4.3]{MM07}.

\subsection{Berezin-Toeplitz quantization on K{\"a}hler orbifolds} \label{pbs4.3}

We apply now the results of Section \ref{pbs4.2} to establish 
the Berezin-Toeplitz quantization on K{\"a}hler orbifolds.
We use the notations and assumptions of that Section.
\\[2pt]
\textbf{\emph{Toeplitz operators on orbifolds.\/}}
We define Toeplitz operators as a family $\{T_p\}$ of linear operators $T_{p}:L^2(X, L^p\otimes E)\longrightarrow L^2(X, L^p\otimes E)$ satisfying the conditions from Definition \ref{toe-def}. 

For any section $f\in \cC^{\infty}(X,\End(E))$,
  the  \emph{Berezin-Toeplitz quantization} of $f$ is defined by
\begin{equation}\label{toe6.3}
T_{f,\,p}:L^2(X,L^p\otimes E)\longrightarrow L^2(X,L^p\otimes E)\,,
\quad T_{f,\,p}=P_p\,f\,P_p\,.
\end{equation} 

Now, by the same argument as in Lemma \ref{toet2.1}, we get
\begin{lemma} \label{toet6.1}
For any $\varepsilon>0$ and any $l,m\in\N$ there exists $C_{l,m,{\varepsilon}}>0$ such that 
\begin{equation} \label{toe6.4}
|T_{f,\,p}(x,x')|_{\cC^m(X\times X)}\leqslant C_{l,m,{\varepsilon}}p^{-l}
\end{equation}
for all $p\geqslant 1$ and all $(x,x')\in X\times X$ 
with $d(x,x')>\varepsilon$, where the $\cC^m$-norm is induced by $\nabla^L,\nabla^E$ and $h^L,h^E,g^{TX}$.
\end{lemma} 

As in Section \ref{s3.3} we obtain next the asymptotic expansion of 
the kernel $T_{f,\,p}(x,x')$ in a neighborhood of the diagonal. 

We need to introduce the appropriate analogue of the condition introduced in the Notation \ref{noe2.7}
 in the orbifold case, in order to take into account the group action 
associated to an orbifold chart.
Let $\{\Theta_p\}_{p\in\N}$ be a sequence of linear operators 
$\Theta_p: L^2(X,L^p\otimes E)\longrightarrow L^2(X,L^p\otimes E)$ 
with smooth kernel $\Theta_p(x,y)$ with respect to $dv_X(y)$.

\begin{condition}\label{coe2.71}
Let $k\in\N$. Assume that for every open set $U\in\mathcal{U}$ and 
every orbifold chart
$(G_U,\widetilde{U})\stackrel{\tau_U}{\longrightarrow} U$, 
there exists a sequence of kernels 
$\{\wi{\Theta}_{p, U}(\wi{x},\wi{x}^{\,\prime})\}_{p\in\N}$ and a family 
$\{Q_{r,\,x_0}\}_{0\leqslant r\leqslant k,\,x_0\in X}$ such that 
\begin{itemize}
\item[(a)] $Q_{r,\,x_0}\in \End( E)_{x_0}[\wi{Z},\wi{Z}^{\prime}]$\,, 
\item[(b)] $\{Q_{r,\,x_0}\}_{r\in\N,\,x_0\in X}$ is smooth with respect 
to the parameter $x_0\in X$,
\item[(c)] for every fixed $\var''>0$ and every 
$\wi{x},\wi{x}^{\,\prime}\in \wi{U}$ the following holds
\begin{equation} \label{toe6.5}
\begin{split}
& (g,1)\wi{\Theta}_{p, U}(g^{-1}\wi{x},\wi{x}^{\,\prime})=
(1,g^{-1})\wi{\Theta}_{p, U}(\wi{x},g\wi{x}^{\,\prime})
\quad \text{for any } \, \, g\in G_{U}\; \text{(cf. \eqref{pb4.22a})}, \\
&\wi{\Theta}_{p, U}(\wi{x},\wi{x}^{\,\prime})= \cO(p^{-\infty}) \quad
 \quad\text{for}\, \,   d(x,x^\prime)>\var'',\\
&\Theta_{p}(x,x^\prime)
= \sum_{g\in G_{U}} (g,1)\wi{\Theta}_{p, U}(g^{-1}\wi{x},\wi{x}^{\,\prime})
+ \cO(p^{-\infty}),
\end{split}\end{equation}
and moreover, for every relatively compact open subset $\wi{V}\subset \wi{U}$, the relation
\begin{equation} \label{toe6.6}
p^{-n}\, \wi{\Theta}_{p,U,\wi{x}_0}(\wi{Z},\wi{Z}^\prime)\cong \sum_{r=0}^k 
(Q_{r,\,\wi{x}_0} \cP_{\wi{x}_0})
(\sqrt{p}\wi{Z},\sqrt{p}\wi{Z}^{\prime})p^{-\frac{r}{2}}
+\mO(p^{-\frac{k+1}{2}}),\:\: \text{for $\wi{x}_0\in \wi{V}$},
\end{equation}
holds in the sense of \eqref{toe2.7}. 
\end{itemize}
\end{condition}
\begin{notation}\label{noe2.71}
If the sequence $\{\Theta_p\}_{p\in\N}$ satisfies Condition \ref{coe2.71},
 we write
\begin{equation} \label{toe6.7}
p^{-n}\, \Theta_{p,\,x_0}(Z,Z^\prime)\cong \sum_{r=0}^k 
(Q_{r,\,x_0} \cP_{x_0})(\sqrt{p}Z,\sqrt{p}Z^{\prime})p^{-\frac{r}{2}}
+\mO(p^{-\frac{k+1}{2}})\,.
\end{equation}
\end{notation}
Note that although the Notations \ref{noe2.71} and \ref{noe2.7} 
are formally similar, they have different meaning.

\begin{lemma}\label{toet6.2} The smooth family
$Q_{r,\,x_0}\in \End( E)_{x_0}[\wi{Z},\wi{Z}^{\prime}]$ 
in Condition \ref{coe2.71} is uniquely determined by $\Theta_p$.
\end{lemma}
\begin{proof} Clearly, for $W\subset U$, the restriction of $\wi{\Theta}_{p,U}$ 
to $\wi{W} \times \wi{W}$ verifies \eqref{toe6.5}, thus we can take 
$\wi{\Theta}_{p,W}=\wi{\Theta}_{p,U}|_{\wi{W} \times \wi{W}}$.
Since $G_U$ acts freely on $\tau_U^{-1}(U_{\rm reg})\subset \wi{U}$, 
we deduce from \eqref{toe6.5} and \eqref{toe6.6} that 
\begin{align}\label{toe6.8}
\Theta_{p,\,x_0}(Z,Z^\prime)= \wi{\Theta}_{p,\,U,\,\wi{x}_0}(\wi{Z},\wi{Z}^\prime)
+ \cO(p^{-\infty})\,,
\end{align}
for every $x_0\in U_{\rm reg}$
and $|\wi{Z}|,|\wi{Z}^\prime|$ small enough.
We infer from  \eqref{toe6.6} and \eqref{toe6.8} that
$Q_{r,\,x_0}\in \End( E)_{x_0}[\wi{Z},\wi{Z}^{\prime}]$ 
is uniquely determined for $x_0\in X_{\rm reg}$\,. 
Since $Q_{r,\,x_0}$ depends smoothly on $x_0$, 
its lift to $\wi{U}$ is smooth. Since the set $\tau_U^{-1}(U_{\rm reg})$ 
is dense in $\wi{U}$, we see that
the smooth family $Q_{r,\,x_0}$ is uniquely determined by $\Theta_p$.
\end{proof}

\begin{lemma}\label{toet6.3} There exist polynomials 
$J_{r,\,x_0}, Q_{r,x_0}(f)$ $\in \End(E)_{x_0}[\wi{Z},\wi{Z}^{\prime}]$ so that  
Theorem \ref{tue17}, Lemmas \ref{toet2.1}, \ref{toet2.3} and \eqref{toe2.20}
still hold under the notation \eqref{toe6.7}.
Moreover, 
\begin{equation}\label{toe6.15}
J_{0,\,x_0}= \Id_E, \quad J_{1,\,x_0}=0.
\end{equation}
\end{lemma}
\begin{proof}   
The analogues of Theorems \ref{tue16}-\ref{tue17} for 
the current situation and \eqref{pb4.22a}, \eqref{pb4.22} show that
Theorem \ref{tue17} and Lemmas \ref{toet2.1}, \ref{toet2.3} still hold under 
the notation \eqref{toe6.7}.
By \eqref{bk2.30}, we have $\mO_1=0$. 
Hence \eqref{bk2.31} entails \eqref{toe6.15}.
Moreover, \eqref{pb4.21} implies
\begin{align}\label{toe6.17}
T_{f,\,p}(x,x^\prime)= \int_{X} F(D_p)(x,x'')f(x'')
 F(D_p)(x'',x^\prime) dv_X(x'') + \cO(p^{-\infty}).
\end{align}
Therefore, we deduce from \eqref{pb4.9}, \eqref{pb4.22a},
\eqref{pb4.22} and \eqref{toe6.17} that Lemmas \ref{toet2.3} 
and \eqref{toe2.20} still hold under the notation \eqref{toe6.7}.
\end{proof}

We have therefore orbifold asymptotic expansions for the Bergman
and Toeplitz kernels, analogues 
to those for smooth manifolds. Following the strategy used in \S \ref{s3.6} 
we can prove a characterization of Toeplitz operators as in Theorem \ref{toet3.1} (see \cite[Th.\,6.11]{MM08b}). 

Proceeding as in \S \ref{s3.6} we can show that the set of Toeplitz operators 
on a compact orbifold is 
closed under the composition of operators, so forms an algebra. 
\begin{theorem}[{\cite[Th.\,6.13]{MM08b}}]\label{toet6.7}
Let $(X,\om)$ be a compact K{\"a}hler orbifold and let $(L,h^L)$ be 
a holomorphic Hermitian proper orbifold line bundle satisfying
the prequantization condition \eqref{prequantum}.
Let $(E,h^E)$ be an arbitrary holomorphic Hermitian proper orbifold 
vector bundle on $X$.

Consider $f,g\in\cC^\infty(X,\End(E))$. 
Then the product of the Toeplitz operators 
$T_{f,\,p}$ and  $T_{g,\,p}$ is a Toeplitz operator, 
more precisely, it admits an asymptotic expansion
in the sense of \eqref{toe2.3},
where $C_r(f,g)\in\cC^\infty(X,\End(E))$
and $C_r$ are bi-differential operators defined locally as in \eqref{toe4.2b}
on each covering $\wi{U}$ of an orbifold chart
 $(G_U,\widetilde{U})\stackrel{\tau_U}{\longrightarrow}U$. 
In particular $C_0(f,g)=fg$.

If $f,g\in\cC^\infty(X)$, then \eqref{toe4.4} holds.

 Relation \eqref{toe4.17} also holds for any $f\in\cC^\infty(X,\End(E))$.
\end{theorem}

\begin{remark}\label{toet4.31} As in Remark \ref{toer1}, 
Theorem \ref{toet6.7} shows that
on every compact K{\"a}hler orbifold $X$ 
admitting a prequantum line bundle $(L,h^L)$,
we can define in a canonical way an associative star-product
\begin{equation}\label{toe6.81}
f*_{\hbar}g=\sum_{l=0}^\infty \hbar^{l}C_l(f,g)\in\cC^\infty(X,\End(E))[[\hbar]]
\end{equation}
for every $f,g\in \cC^\infty(X,\End(E))$,
called the \emph{Berezin-Toeplitz star-product}\/.
Moreover, $C_l(f,g)$ are bi-differential operators 
defined locally as in the smooth case.
\end{remark}

\section{Quantization of symplectic manifolds}

We will briefly describe in this Section how to generalize the ideas used before in the K\"ahler case in order to study the Toeplitz operators and Berezin-Toeplitz quantization for symplectic manifolds. For details we refer the reader to \cite{MM07,MM08b}.
We recall in Section 4.1 the definition of the spin$^c$ Dirac operator and formulate the spectral gap property 
for prequantum line bundles. In Section 4.2 we state the asymptotic expansion of the composition of Toeplitz operators.

\subsection{Spectral gap of the spin$^c$ Dirac operator}
We will first show that in the general symplectic case the kernel of the spin$^c$ operator is a good substitute for the space of holomorphic sections used in K\"ahler quantization. 

Let $(X,\omega)$ be a compact symplectic manifold, $\dim_\mathbb{R} X=2n$, with compatible almost complex structure $J:TX\to TX$. Let $g^{TX}$ be the associated Riemannian metric compatible with $\omega$, i.e., $g^{TX}(u,v)=\omega(u,Jv)$. Let $(L,h^L,\nabla^L)\to X$ be Hermitian line bundle, endowed with a Hermitian metric $h^L$ and a Hermitian connection $\nabla^L$, whose curvature is $R^L=(\nabla^L)^2$. We assume that the \emph{prequantization condition} \eqref{prequantum} is fulfilled.
Let $(E,h^E,\nabla^E)\to X$ be a Hermitian vector bundle. 
We will be concerned with asymptotics in terms of high tensor powers $L^p\otimes E$, when $p\to\infty$,
that is, we consider the semi-classical limit $\hbar=1/p\to 0$.

Let $\nabla^{\rm det}$ be the connection on $\det(T^{(1,0)}X)$ induced by the projection of the Levi-Civita connection $\nabla^{TX}$ on $T^{(1,0)}X$.
Let us consider the Clifford connection $\nabla^{\text{Cliff}}$ on $\Lambda^{\scriptscriptstyle{\bullet}}(T^{*(0,1)}X)$ associated to $\nabla^{TX}$ and to the connection $\nabla^{\rm det}$ on $\det(T^{(1,0)}X)$ (see e.g.\ \cite[\S\,1.3]{MM07}). The connections $\nabla^L$, $\nabla^E$ and $\nabla^{\text{Cliff}}$ induce the connection 
\[
\nabla_p=\nabla^{\text{Cliff}}\otimes\operatorname{Id}+\operatorname{Id}\otimes\nabla^{L^p\otimes E}\quad
\text{on $\Lambda^{\scriptscriptstyle{\bullet}}(T^{*(0,1)}X)\otimes L^p\otimes E$.}
\]
The \emph{spin$^c$ Dirac operator} is defined by 
\begin{equation}\label{defDirac}
D_p=\sum_{j=1}^{2n}\mathbf{c}(e_j)\nabla_{p,e_j}:
\Omega^{0,\scriptscriptstyle{\bullet}}(X,L^p\otimes E)\longrightarrow
\Omega^{0,\scriptscriptstyle{\bullet}}(X,L^p\otimes E)\,.
\end{equation}
where $\{e_j\}_{j=1}^{2n}$ local orthonormal frame of $TX$ and $\mathbf{c}(v)=\sqrt{2}({\overline v^\ast_{1,0}}\wedge-i_{v_{\,0,1}})$ 
is the Clifford action of $v\in TX$. Here we use the decomposition $v=v_{\,1,0}+v_{\,0,1}$, 
$v_{\,1,0}\in T^{(1,0)}X$, $v_{\,0,1}\in T^{(0,1)}X$.
 

If $(X,J,\omega)$ is K\"ahler then $D_p=\sqrt{2}(\overline{\partial}+\overline{\partial}^{\,*})$ so $\Ker(D_p)=H^0(X,L^p\otimes E)$ for $p\gg1$. The following result shows that $\Ker(D_p)$ has all semi-classical properties of $H^0(X,L^p\otimes E)$.
The proof is based on a 
direct application of the Lichnerowicz formula for $D_p^2$.
Note that the metrics $g^{TX}$, $h^L$ and $h^E$ induce an $L^2$-scalar product 
on $\Omega^{0,\scriptscriptstyle{\bullet}}(X,L^p\otimes E)$, whose completion is denoted 
$(\Omega^{0,\scriptscriptstyle{\bullet}}_{(2)}(X,L^p\otimes E),\|\cdot\|_{L^2})$. 

\begin{theorem}[{\cite[Th.\,1.1,\,2.5]{MM02}, \cite[Th.\,1.5.5]{MM07}}]\label{specDirac}
There exists $C>0$ such that for any $p\in\mathbb{N}$ and any 
$s\in\bigoplus_{k>0}\Omega^{0,k}(X,L^p\otimes E)$ we have
\begin{equation}\label{main1}
\Vert D_{p}s\Vert^2_{L^2}\geq(4\pi p-C)\Vert s\Vert^2_{L^2}\, .
\end{equation}
Moreover, 
the spectrum of $D^2_p$ verifies 
\begin{align}\label{diag5}
&\spec (D^2_p) \subset \{0\}\cup [4\pi p-C,+\infty[\,.
\end{align}
\end{theorem}
\noindent
By the Atiyah-Singer index theorem we have for $p\gg1$
\begin{equation}\label{as}
\dim \Ker(D_p)=\int_{X}\operatorname{Td}(T^{(1,0)}X)\operatorname{ch}(L^{p}\otimes E)
=\rank (E)\,\frac{p^n}{n!}\int_X \omega^n+\mO(p^{n-1}) \,. 
\end{equation}
Theorem \ref{specDirac} shows the forms in $\Ker(D_p)$ concentrate asymptotically in the $L^2$ sense on their zero-degree component and \eqref{as} shows that $\dim\Ker(D_p)$ is a polynomial 
in $p$ of degree $n$, as in the holomorphic case. 
\subsection{Toeplitz operators in spin$^c$ quantization}
Let us introduce the orthogonal projection
$P_p:\Omega^{0,\scriptscriptstyle{\bullet}}_{(2)}(X,L^p\otimes E)\longrightarrow\Ker(D_p)$, called the Bergman projection in analogy to the K\"ahler case. Its integral kernel is called \emph{Bergman kernel}\/. The \emph{Toeplitz operator} with symbol $f\in\cC^{\infty}(X,\End(E))$ is
\[
T_{f,p}:\Omega^{0,\scriptscriptstyle{\bullet}}_{(2)}(X,L^p\otimes E)\to \Omega^{0,\scriptscriptstyle{\bullet}}_{(2)}(X,L^p\otimes E)\,,\quad T_{f,p}= P_{p}fP_{p}
\]
In analogy to the K\"ahler case we define a (generalized) {\em Toeplitz operator} 
is a sequence $(T_p)$ of linear operators
$T_{p}\in\End(\Omega^{0,\scriptscriptstyle{\bullet}}_{(2)}(X,L^p\otimes E))$
verifying $T_{p}=P_p\,T_p\,P_p$\,,
such that there exist a sequence $g_l\in\cC^\infty(X,\operatorname{End}(E))$ with the property that
for all $k\geq0$, there exists $C_k>0$ so that \eqref{toe2.3} is fulfilled.

A basic fact is that the Bergman kernel $P_p(\cdot,\cdot)$ of the Dirac operator has an asymptotic expansion similar to Theorems \ref{tue16} and \ref{tue17}. This was shown by Dai-Liu-Ma in \cite[Prop.\,4.1 and Th.\,4.18$^\prime$]{DLM04a} (see also
\cite[Th.\,8.1.4]{MM07}).
By the Bergman kernel expansion of Dai-Liu-Ma we obtain the expansion of the integral kernels of $T_{f,\,p}$\,, 
similar to Theorem \ref{toet2.3}.
Moreover, the characterization of Toeplitz operators in terms of the off-diagonal asymptotic expansion of their integral kernels, formulated in Theorem \ref{toet3.1}, holds also in the symplectic case (cf.\ \cite[Th.\,4.9]{MM08b}, \cite[Lemmas\,7.2.2,\,7.2.4,\,Th.\,7.3.1]{MM07}).
We obtain thus the symplectic analogue of Theorem \ref{toet4.1b}.
\begin{theorem}[{\cite[Th.\,1.1]{MM08b}, 
\cite[Th.\,8.1.10]{MM07}}]\label{toet4.1}
Let $f,g\in\cC^\infty(X,\End(E))$. The composition $(T_{f,\,p}\circ T_{g,\,p})$ 
is a Toeplitz operator, i.e., 
\begin{equation}\label{toe4.2}
T_{f,\,p}\circ T_{g,\,p}=\sum^\infty_{r=0}p^{-r}T_{C_r(f,\,g),\,p}+\mO(p^{-\infty}),
\end{equation} 
where $C_r$ are bi-differential operators, 
$C_0(f,g)=fg$ and $C_r(f,g)\in\cC^\infty(X,\End(E))$.
Let $f,g\in\cC^\infty(X)$ and let $\{\cdot,\cdot\}$ be the Poisson bracket on $(X,2\pi\omega)$, defined as in 
\eqref{toe4.1a}. Then 
\begin{equation}\label{toe4.3b}
C_1(f,g)-C_1(g,f)= \sqrt{-1}\{f,g\} \Id_E,
\end{equation} 
and therefore
\begin{equation}\label{toe4.4a}
\big[T_{f,\,p}\,,T_{g,\,p}\big]=\frac{\sqrt{-1}}{\, p}T_{\{f,\,g\},\,p}+\mO(p^{-2}).
\end{equation} 
\end{theorem}
Thus the construction of the Berezin-Toeplitz star-product can be carried out also in the case of symplectic manifolds. Namely, for $f,g\in\cC^\infty(X,\End(E))$ we set
$f*_{\hbar}g:=\sum_{k=0}^\infty C_k(f,g) \hbar^{k}\in\cC^\infty(X,\End(E))[[\hbar]]$,
where $C_{r}(f,g)$ are determined by \eqref{toe4.2}. Then $*_{\hbar}$ is an associative star product.

\providecommand{\bysame}{\leavevmode\hbox to3em{\hrulefill}\thinspace}


\end{document}